\documentclass[a4paper,11pt]{amsart}
\usepackage{tipa}
\usepackage{txfonts}
\usepackage{amssymb}
\usepackage{amsmath}
\usepackage{mathrsfs}
\usepackage{amsmath,amssymb,amsthm,latexsym,amscd,mathrsfs}
\usepackage{indentfirst}
\usepackage{graphicx}
\usepackage{booktabs}
\usepackage{array}
\usepackage{chemarrow}
\newcommand{\mi}{\mathrm{i}}
\newcommand{\mj}{\mathrm{j}}
\newcommand{\mk}{\mathrm{k}}

\newcommand{\ROM}[1]{\mathrm{\uppercase\expandafter{\romannumeral#1}}}
\theoremstyle{definition}

\newtheorem{thm}{Theorem}[section]
\newtheorem{lem}{Lemma}[section]

\newtheorem{cor}{Corollary}[section]
\newtheorem{examp}{Example}[section]
\newtheorem{rem}{Remark}[section]
\newtheorem{prop}{Proposition}[section]

\newtheorem{prob}[thm]{Problem}
\newtheorem{conj}[thm]{Conjecture}

\newtheorem{ack}{Acknowledgements}   
\makeatletter


\title[Isoparametric foliations, a problem of Eells-Lemaire and conjectures of Leung]{\textbf{Isoparametric foliations, a problem of Eells-Lemaire and conjectures of Leung}}

\author[C. Qian]{Chao Qian}\address{School of Mathematics and Statistics, Beijing Institute of Technology, Beijing 100081, P.R. China}
\email{6120150035@bit.edu.cn}

\author[Z. Z. Tang]{Zizhou Tang}\address{School of Mathematical Sciences, Laboratory of Mathematics and Complex Systems, Beijing Normal University, Beijing 100875, P.R.China}\email{zztang@bnu.edu.cn}
\thanks {The first author was partially supported by NSFC (No. 11401560, No. 11571339), and the second author(corresponding author) was partially supported by NSFC 
(No. 11331002).}

\thanks{}

\thanks{}

\date{}
\begin{document}

\maketitle
\begin{abstract}
In this paper, two sequences of minimal isoparametric hypersurfaces are constructed via representations
of Clifford algebras. Based on these, we give estimates on eigenvalues of the Laplacian
of the focal submanifolds of isoparametric hypersurfaces in unit spheres. This
improves results of \cite{TY13} and \cite{TXY14}.

Eells and Lemaire \cite{EL83} posed a problem to characterize the compact Riemannian manifold $M$ for which
there is an eigenmap from $M$ to $S^n$. As another application of our constructions, the focal maps
give rise to many examples of eigenmaps from minimal isoparametric hypersurfaces to unit spheres.

Most importantly, by investigating the second fundamental forms of focal submanifolds of isoparametric hypersurfaces in unit spheres, we provide infinitely many counterexamples to two conjectures of Leung \cite{Le91} (posed in 1991) on minimal submanifolds in unit spheres. Notice that these conjectures of Leung have been proved in the case that the normal connection is flat \cite{HV01}.
\end{abstract}

\section{Introduction}
Let $N$ be a connected complete Riemannian manifold. A non-constant
smooth function $f$ on $N$ is called \emph{transnormal}, if there
exists a smooth function $b:\mathbb{R}\rightarrow\mathbb{R}$ such
that the gradient of $f$ satisfies $|\nabla f|^2=b(f)$.
Moreover, if there exists another function
$a:\mathbb{R}\rightarrow\mathbb{R}$ so that the Laplacian of $f$ satisfies $\triangle f=a(f)$, then $f$ is said to be
\emph{isoparametric}. Each regular level hypersurface of $f$ is then called an
\emph{isoparametric hypersurface}. It was proved by Wang (see \cite{Wa87})
that each singular level set is also a smooth
submanifold ( not necessarily connected ), the so-called \emph{focal submanifold}.
The whole family of isoparametric hypersurfaces together with the focal
submanifolds form a singular Riemannian foliation,
which is called the \emph{isoparametric foliation}. For recent study of isoparametric functions
on general Riemannian manifolds, especially on exotic spheres, see \cite{GT13} and \cite{QT15}.

E. Cartan was the first to give a systematic study on isoparametric hypersurfaces in real
space forms and proved that an isoparametric hypersurface is exactly
a hypersurface with constant principal curvatures in these cases.
For the spherical case (the most interesting and complicated case), Cartan obtained
the classification result under the assumption that the number of the distinct principal curvatures is at most $3$.
Later, H. F. M\"{u}nzner \cite{Mu80} extended widely Cartan's work.
To be precise, given an isoparametric hypersurface $M^n$ in $S^{n+1}(1)$, let $\xi$ be
a unit normal vector field along $M^n$ in $S^{n+1}(1)$, $g$ the number
of distinct principal curvatures of $M$, $\cot \theta_{\alpha}~
(\alpha=1,...,g;~ 0<\theta_1<\cdots<\theta_{g} <\pi)$ the principal
curvatures with respect to $\xi$ and $m_{\alpha}$ the multiplicity
of $\cot \theta_{\alpha}$. M\"{u}nzner proved that $m_{\alpha}=m_{\alpha+2}$
(indices mod $g$),
$\theta_{\alpha}=\theta_1+\frac{\alpha-1}{g}\pi$ $(\alpha = 1,...,
g)$, and there exists
a homogeneous polynomial $F: \mathbb{R}^{n+2}\rightarrow \mathbb{R}$ of degree $g$,
the so-called\emph{ Cartan-M\"{u}nzner polynomial}, satisfying
\begin{equation}\label{ab}
\left\{ \begin{array}{ll}
|\tilde{\nabla} F|^2= g^2r^{2g-2}, \nonumber\\
~~~~\tilde{\triangle} F~~=\frac{m_2-m_1}{2}g^2r^{g-2},\nonumber
\end{array}\right.
\end{equation}
where $r=|x|$, $m_1$ and $m_2$ are the two multiplicities, and $\tilde{\nabla}, \tilde{\triangle}$ are
Euclidean gradient and Laplacian, respectively. Moreover,
M\"{u}nzner obtained the remarkable result that
$g$ must be $1, 2, 3, 4$ or $6$. Since then,
the classification of isoparametric hypersurfaces with $g=4$ or $6$
in a unit sphere has been one of the most challenging problems in
differential geometry.

Recently, due to \cite{CCJ07}, \cite{Im08}, \cite{Ch11} and \cite{Ch13}, an
isoparametric hypersurface with $g=4$ in a unit sphere must be homogeneous
or OT-FKM type(see below) except for the case $(m_1,m_2)=(7,8)$. For $g=6$,
R. Miyaoka \cite{Mi13}, \cite{Mi16} completed
the classification by showing that isoparametric hypersurfaces
in this case are always homogeneous.

To prepare for our results, let us now recall the isoparametric hypersurfaces of
OT-FKM type(c.f. \cite{FKM81}). Given a symmetric Clifford system
$\{P_0,\cdots,P_m\}$ on $\mathbb{R}^{2l}$,
\emph{i.e.}, $P_0, ..., P_m$ are symmetric matrices
satisfying $P_{\alpha}P_{\beta}+P_{\beta}P_{\alpha}=2\delta_{\alpha\beta}I_{2l}$, Ferus, Karcher and
M\"{u}nzner defined a polynomial
$F:\mathbb{R}^{2l}\rightarrow \mathbb{R}$ by
$F(x) = |x|^4 - 2\displaystyle\sum_{\alpha = 0}^{m}{\langle
P_{\alpha}x,x\rangle^2}$. They verified that $f=F|_{S^{2l-1}(1)}$ is
an isoparametric function on $S^{2l-1}(1)$ and each level
hypersurface of $f$ has $4$ distinct constant
principal curvatures with $(m_1, m_2)=(m, l-m-1)$, provided
$m>0$ and $l-m-1> 0$, where $l = k\delta(m)$ $(k=1,2,3,...)$
and $\delta(m)$ is the dimension of an irreducible module of
the Clifford algebra $C_{m-1}$. As usual, for OT-FKM type, we
denote the two focal submanifolds by $M_+=f^{-1}(1)$ and
$M_-=f^{-1}(-1)$, which have codimension $m_1+1$ and $m_2+1$
in $S^{2l-1}(1)$, respectively.

In the first part of the paper, inspired by the OT-FKM construction, for a symmetric Clifford system
$\{P_0,\cdots,P_m\}$ on $\mathbb{R}^{2l}$ with the Euclidean metric $\langle\cdot, \cdot\rangle$, we define $M_i:=\{x\in S^{2l-1}(1)~
|~\langle P_0x, x\rangle=\langle P_1x, x\rangle
=\cdots=\langle P_ix, x\rangle=0\}$, and then we have a sequence
$$M_m=M_+\subset M_{m-1}\subset\cdots\subset M_0
\subset S^{2l-1}(1).$$
For $0 \leq i \leq m-1$, it is natural to define a function $f_i:M_i\rightarrow \mathbb{R}$~~by
~~$f_i(x)=\langle P_{i+1}x, x\rangle$ for $x\in M_i$(see also \cite{TY12'}).

Similarly, by defining $N_i:=\{x\in S^{2l-1}(1)~|~\langle P_0x, x\rangle^2+
\langle P_1x, x\rangle^2+\cdots+\langle P_ix, x\rangle^2=1\}$, we construct
another sequence
$$N_{1}\subset N_2\subset \cdots\subset N_m=M_-\subset S^{2l-1}(1).$$
And for $2 \leq i \leq m$, we define a function $g_i:N_i\rightarrow \mathbb{R}$
by $g_i(x)=\langle P_{i}x, x\rangle$ for $x\in N_i.$
Henceforth, we always regard $M_i$ and $N_i$ as Riemannian manifolds with the induced metric
in $S^{2l-1}(1)$.
One of the main results in this paper is now stated as the following theorem.
\begin{thm}\label{filtration}
Assume the notations as above.
\item[(1).] For $0 \leq i \leq m-1$, the function $f_i:M_i\rightarrow \mathbb{R}$ with
$\mathrm{Im}(f_i)=[-1,~1]$ is an isoparametric function satisfying $$|\nabla f_i|^2=4(1-f_i^2),~
\triangle f_i~~=-4(l-i-1)f_i.$$
For any $c\in (-1,~1)$, the regular level set $\mathcal{U}_c=f_i^{-1}(c)$ has $3$
distinct principal curvatures $-\sqrt{\frac{1-c}{1+c}}$, 0, $\sqrt{\frac{1+c}{1-c}}$ with
multiplicities $l-i-2$, $i+1$, and $l-i-2$ respectively,
w.r.t. the unit normal $\xi=\frac{\nabla f_i}{|\nabla f_i|}$.
For $c=\pm 1$, the two focal submanifolds $\mathcal{U}_{\pm 1}=f_i^{-1}(\pm 1)$ are both
isometric to $S^{l-1}(1)$ and are totally geodesic in $M_i$.

Particularly, we have a minimal isoparametric sequence $M_m\subset M_{m-1}\subset\cdots\subset M_0\subset
S^{2l-1}(1)$, i.e., each $M_{i+1}$ is a minimal isoparametric hypersurface in $M_{i}$
for $0 \leq i \leq m-1$. Moreover, $M_{i+j}$ is minimal in $M_i$.

\item[(2).] Similarly, for $2 \leq i \leq m$, the function $g_i:N_i\rightarrow \mathbb{R}$ with
$\mathrm{Im}(g_i)=[-1,~1]$ is an isoparametric function satisfying $$|\nabla g_i|^2=4(1-g_i^2),~
\triangle g_i~~=-4ig_i.$$
For any $c\in (-1,~1)$, the regular level set $\mathcal{V}_c=g_i^{-1}(c)$ has $3$
distinct principal curvatures $-\sqrt{\frac{1-c}{1+c}}$, 0, $\sqrt{\frac{1+c}{1-c}}$ with
multiplicities $i-1$, $l-i$, and $i-1$ respectively,
w.r.t. the unit normal $\eta=\frac{\nabla g_i}{|\nabla g_i|}$.
For $c=\pm 1$, the two focal submanifolds $\mathcal{V}_{\pm 1}=g_i^{-1}(\pm 1)$ are both
isometric to $S^{l-1}(1)$ and are totally geodesic in $N_i$.

In particular, we get another minimal isoparametric sequence $N_{1}\subset N_2\subset \cdots\subset
N_m\subset S^{2l-1}(1)$, i.e., each $N_{i-1}$ is a minimal isoparametric hypersurface
in $N_i$ for $2 \leq i \leq m$. Moreover, $N_{i}$ is minimal in $N_{i+j}$.
\end{thm}
As a consequence, we have
\begin{cor}\label{fibration}
Assume the notations as in Theorem \ref{filtration}.

(1). For $0 \leq i \leq m-1$, each $M_{i+1}$ fibers over $S^{l-1}$ with fiber
$S^{l-i-2}$.

(2). For $2 \leq i \leq m$, each $N_{i-1}$ fibers over $S^{l-1}$ with fiber $S^{i-1}$.
\end{cor}

\begin{rem}
For $i=m-1$, the first part of the above corollary gives a geometric interpretation of
Lemma $3$ in \cite{Wa88}.
\end{rem}
\begin{rem}
Very recently, using representations of Clifford algebras, M. Radeschi in \cite{Ra14} constructed indecomposable singular Riemannian foliations on round spheres, most of which are non-homogeneous.
\end{rem}
In next part, we will apply the above constructions of isoparametric functions to the estimates
of eigenvalues. Given an $n$-dimensional closed Riemannian manifold $M^n$, recall that the
Laplace-Beltrami operator acting on smooth functions on $M$ is an elliptic operator and has a
discrete spectrum
$\{0=\lambda_0(M)<\lambda_1(M)\leq \lambda_2(M)\leq\cdots\leq \lambda_k(M)\leq\cdots, k
\uparrow\infty\},$
with each eigenvalue counted with its multiplicity.
Following the way in \cite{TY13} and \cite{TXY14},
we acquire the following theorem on eigenvalue estimates, based on
isoparametric foliations constructed in Theorem \ref{filtration}.
\begin{thm}\label{eigenvalue}
Let $\{P_0,\cdots,P_m\}$ be a symmetric Clifford system on $\mathbb{R}^{2l}$.
\item[(1).] For the sequence $M_m\subset M_{m-1}\subset\cdots\subset M_0\subset
S^{2l-1}(1)$, the following inequalities hold

a). $\lambda_k(M_{i})\leq \frac{l-i-2}{l-i-3}\lambda_k(M_{i+1})$
provided that $0 \leq i\leq m-1$ and $l-i-3>0$;

b). $\lambda_k(M_{i+1})\leq 2\lambda_k(S^{l-1}(1))$ provided that $0 \leq i\leq m-1$.

\item[(2).] For the sequence $N_{1}\subset N_2\subset\cdots\subset
N_m\subset S^{2l-1}(1)$, the following inequalities hold

a). $\lambda_k(N_{i})\leq \frac{i-1}{i-2}\lambda_k(N_{i-1})$ provided that $3\leq i\leq m$;

b). $\lambda_k(N_{i-1})\leq 2\lambda_k(S^{l-1}(1))$ provided that $2\leq i\leq m$.
\end{thm}

In the third part, as an unexpected phenomenon, we find the relations between the focal maps
of isoparametric foliations constructed in Theorem \ref{filtration} and harmonic maps. To be more precise,
let $M$ and $N$ be closed Riemannian manifolds,
and $f$ a smooth map from $M$ to $N$. The energy functional $E(f)$ is defined by
$E(f)=\frac{1}{2}\int_M|df|^2dV_M.$
The map $f$ is called \emph{harmonic} if it is
a critical point of the energy functional $E$. We refer to \cite{EL78} and \cite{EL88}
for the background and development of this topic. For $N=S^n(1)$, a map $\varphi:M\rightarrow
S^n(1)$ is called an \emph{eigenmap} if the $\mathbb{R}^{n+1}$-components are
eigenfunctions of the Laplacian of $M$ and all have the same eigenvalue. In particular,
$\varphi$ is a harmonic map. In 1980, Eells and Lemaire ( See p. 70 of \cite{EL83} ) posed
the following

\begin{prob}
Characterize those compact $M$ for which there is an eigenmap $\varphi:M\rightarrow
S^n(1)$ with $\dim(M)\geq n$ ?
\end{prob}
In 1993, Eells and Ratto ( See p. 132 of \cite{ER93} ) emphasized again that it is
quite natural to study the eigenmaps to $S^n(1)$.
As another application of our constructions in Theorem \ref{filtration}, we prove that
\begin{thm}\label{eigenmap}Let $\{P_0,\cdots,P_m\}$ be a symmetric Clifford system
on $\mathbb{R}^{2l}$.

(1). For $0\leq i \leq m-1$, both of the focal maps $\phi_{\pm \frac{\pi}{4}}: M_{i+1} \rightarrow
\mathcal{U}_{\pm 1}\cong S^{l-1}(1)$ defined by
$$\phi_{\pm \frac{\pi}{4}}(x)=\frac{1}{\sqrt{2}}(x\pm P_{i+1}x),~x \in M_{i+1},$$
are submersive eigenmaps with the same eigenvalue $2l-i-3$.

(2). For $2\leq i \leq m$, both of the focal maps $\psi_{\pm \frac{\pi}{4}}:N_{i-1}\rightarrow
\mathcal{V}_{\pm 1}\cong S^{l-1}(1)$ defined by
$$\psi_{\pm \frac{\pi}{4}}(x)=\frac{1}{\sqrt{2}}(x\pm P_{i}x),~x \in N_{i},$$
are submersive eigenmaps with the same eigenvalue $l+i-2$.

\end{thm}

Meanwhile, the case for isoparametric foliations on unit spheres is also considered.
Given an isoparametric hypersurface $M^n$ (not necessarily minimal) in $S^{n+1}(1)$ and a smooth field
$\xi$ of unit normals to $M$, for each $x\in M$ and $\theta\in \mathbb{R}$, one has a map
$\varphi_{\theta}: M^n\rightarrow S^{n+1}(1)$ by
$$\varphi_{\theta}(x)=\cos \theta~ x +\sin \theta~ \xi(x).$$
If $\theta\neq \theta_{\alpha}$ for any $\alpha=1,...,g$,
$\varphi_{\theta}$ is a parallel hypersurface to $M$. If $\theta= \theta_{\alpha}$ for some
$\alpha=1,...,g$, i.e., $\cot \theta=\cot \theta_{\alpha}$ is a principal curvature
of $M$, $\varphi_{\theta}$ is not an immersion,
actually a focal submanifold of codimension
$m_{\alpha}+1$ in $S^{n+1}(1)$. And the map $\varphi_{\theta}$
from $M$ to a focal submanifold is said to be a \emph{focal map}.
M\"{u}nzner \cite{Mu80} asserted there are only two distinct
focal submanifolds, and every isoparametric hypersurface is a tube of constant
radius over each focal submanifold. Denote by $M_+$ and $M_-$ the focal submanifolds in
$S^{n+1}(1)$ with codimension $m_1+1$ and $m_2+1$, respectively. However, there are more than two focal maps.
\begin{prop}\label{focal map}
Let $M$ be a closed isoparametric hypersurface in a unit sphere. Then every
focal map from $M$ to its focal submanifolds $M_{+}$ or $M_-$ is harmonic.
\end{prop}

Furthermore,  we will investigate the stability of harmonic maps constructed in Theorem
\ref{eigenmap} and Proposition \ref{focal map} as well.

The last part of the paper will be concerned with the pinching problem for minimal submanifolds in unit
spheres. Let $W^n$ be a closed Riemannian manifold minimally immersed in
$S^{n+p}(1)$. Let $B$ be the second fundamental form
and define an extrinsic quantity $\sigma(W)=\max\{~|B(X,X)|^2~|~X\in TM,~|X|=1\}$.

In 1986, H. Gauchman \cite{Ga86} established a well known rigidity theorem which
states that if $\sigma (W)<1/3$, then the submanifold $W$ must be totally geodesic.
When the dimension $n$ of $W$ is even, the rigidity theorem above is optimal.
As presented in \cite{Ga86}, there exist minimal submanifolds
in unit spheres which are not totally geodesic, with $|B(X,X)|^2\equiv 1/3$ for any
unit tangent vector $X$.  When the dimension $n$ of $W$ is odd and $p>1$, the conclusion still holds
under a weaker assumption $\sigma(W)\leq \frac{1}{3-2/n}$. It is remarkable that Gauchman's rigidity theorem 
has been generalized to the case of submanifolds with parallel mean curvature in \cite{XFX06}.

In 1991, P. F. Leung \cite{Le91} proved that if $n$ is odd, a closed minimally
immersed submanifold $W^n$ with $\sigma(W)\leq \frac{n}{n-1}$ is totally geodesic
provided that the normal connection is flat. Based on this fact, he proposed the
following
\begin{conj}\label{weak}  If $n$ is odd, $W^n$ is minimally immersed in $S^{n+p}(1)$
with $\sigma(W) \leq\frac{n}{n-1}$, then $W$ is homeomorphic to $S^n$.
\end{conj}

By investigating the second fundamental form of the Clifford minimal hypersurfaces
in unit spheres, Leung also posed the following stronger

\begin{conj}\label{strong}  If $n$ is odd and $W^n$ is minimally immersed in $S^{n+p}(1)$
with $\sigma(W) <\frac{n+1}{n-1}$, then $W$ is homeomorphic to $S^n$.
\end{conj}
For minimal submanifolds in unit spheres with flat normal connections,
Conjecture \ref{strong} was proved by T. Hasanis and T. Vlachos \cite{HV01}.
In fact, they showed that the condition $\mathrm{Ric}(W)$$ >\frac{n(n-3)}{n-1}$ is equivalent
to the inequality $\sigma(W)< \frac{n+1}{n-1}$. Thus in
the case that the normal connection is flat, Conjecture \ref{strong} follows
from Theorem B in \cite{HV01}.

Recall that the examples with even dimensions and $\sigma(W)= 1/3$ given
in \cite{Ga86} originated from the Veronese embeddings of the projective
planes $\mathbb{R}P^2$, $\mathbb{C}P^2$, $\mathbb{H}P^2$ and $\mathbb{O}P^2$ in $S^4(1)$, $S^7(1)$, $S^{13}(1)$ and
$S^{25}(1)$, respectively. Observe that those Veronese submanifolds are
just the focal submanifolds of isoparametric hypersurfaces
in unit spheres with $g=3$. Hence, it is very natural for us to consider
the case with $g=4$.
\begin{thm}\label{counter example}
Let $M^n$ be an isoparametric hypersurface in $S^{n+1}(1)$ with $g=4$ and
multiplicities $(m_1, m_2)$,  and denote by $M_+$ and $M_-$ the focal
submanifolds of $M^n$ in $S^{n+1}(1)$ with dimension $m_1+2m_2$ and $2m_1+m_2$
respectively. Then $M_{\pm}$ are minimal in $S^{n+1}(1)$ with
$\sigma(M_{\pm})=1$. However, $M_{\pm}$ are not homeomorphic to the spheres.
\end{thm}

\begin{rem}
1). If $m_1$ is odd, $M_+\subset S^{n+1}(1)$ in Theorem \ref{counter example} is
a counterexample to Conjecture \ref{weak} and Conjecture \ref{strong}. Similarly,
if $m_2$ is odd, $M_- \subset S^{n+1}(1)$ can be also served as a counterexample to both
of the conjectures.

2). It is not difficult to show directly that the normal connections of those focal
submanifolds in unit spheres are non-flat despite the dimensions.
\end{rem}

The present paper is organized as follows. In Section 2, we will prove Theorem
\ref{filtration} and give a detailed investigation into the geometric properties
of isoparametric foliations we constructed. Based on Section 2, Theorem
\ref{eigenvalue} will be proved in Section 3. In Section 4, we are mainly
concerned with the harmonicity of the focal maps. Moreover, the stability
will be studied as well. Finally in Section 5, infinitely many counterexamples will be
provided to the conjectures of Leung.

\section{Constructions of isoparametric foliations}
The aim of this section is to give a complete proof of Theorem \ref{filtration}.
For convenience, the proof will be divided into several lemmas.

\begin{lem}\label{isoparametric}
Let $\{P_0,\cdots,P_m\}$ be a symmetric Clifford system on $\mathbb{R}^{2l}$ and $M_i, N_i$ defined in the introduction. For $0 \leq i \leq m-1$,
the function $f_i:M_i\rightarrow \mathbb{R}, x\mapsto \langle P_{i+1}x, x\rangle$, is an isoparametric function with
$\mathrm{Im}(f_i)=[-1,~1]$ and satisfies
$$|\nabla f_i|^2=4(1-f_i^2),~\triangle f_i~~=-4(l-i-1)f_i.$$
While for $2 \leq i \leq m$, the function $g_i:N_i\rightarrow \mathbb{R}, x\mapsto\langle P_ix, x\rangle$, is also an isoparametric function with
$\mathrm{Im}(g_i)=[-1,~1]$ and satisfies $$|\nabla g_i|^2=4(1-g_i^2),~
\triangle g_i~~=-4ig_i.$$
\end{lem}
\begin{proof}
First, the function $f_i$ on $M_i$ will be considered. According to \cite{FKM81}, $M_i$ is a smooth submanifold in
$S^{2l-1}(1)\subset \mathbb{R}^{2l}$ with $\mathrm{dim}M_i=2l-i-2$. As we defined in the introduction, it is convenient to regard the function $f_i$ as
the restriction of
the function $F_i: \mathbb{R}^{2l}\rightarrow\mathbb{R}, x\mapsto \langle P_{i+1}x, x\rangle$ to $M_i$.
Henceforth, we will denote the covariant derivatives and Laplacians of $M_i$ and $\mathbb{R}^{2l}$ by
$\nabla, \triangle$ and $\tilde{\nabla}, \tilde{\triangle}$ respectively.
For any $x\in M_i$, $\nu_x M_i$, the normal space of $M_i$ in $\mathbb{R}^{2l}$ at $x$, is equal to
$\mathrm{Span}\{P_0x,\cdots,P_ix, x\}$.
By a direct computation and the property of the symmetric Clifford system, we have $\tilde{\nabla}F_i=2P_{i+1}x$, and
$\langle \tilde{\nabla}F_i, P_{\alpha}x \rangle=0$ for any $0\leq \alpha\leq i$. And it follows that
\begin{eqnarray}
\nabla f_i&=&\tilde{\nabla}F_i-\langle\tilde{\nabla}F_i, x\rangle x-\sum_{\alpha=0}^{i}\langle\tilde{\nabla}F_i, P_{\alpha}x\rangle
P_{\alpha}x\nonumber\\
&=&\tilde{\nabla}F_i-\langle\tilde{\nabla}F_i, x\rangle x\nonumber\\
&=&2P_{i+1}x-2\langle P_{i+1}x, x\rangle x.\nonumber
\end{eqnarray}
Hence $|\nabla f_i|^2=4(1-f_i^2)$. For the Laplacian of $f_i$, we note
that $$\mathrm{Hess}f_i(X,Y)=\widetilde{\mathrm{Hess}}F_i(X,Y)+B(X,Y)(F_i)$$
for $X, Y \in T_xM_i$,
where $B$ is the second fundamental form of $M_i$ in $\mathbb{R}^{2l}$, and $\mathrm{Hess}f_i$,
$\widetilde{\mathrm{Hess}}F_i$ are the Hessians of $f_i$, $F_i$, respectively.
Choose an orthonormal basis $\{e_{a}\}_{a=1}^{2l-i-2}$ for $T_xM_i$. In virtue of
M\"{u}nzner \cite{Mu80}, $M_i$ is minimal in $S^{2l-1}(1)$, i.e., $\sum_{a}\langle B(e_a, e_a), P_{\alpha}x\rangle=0$,
for $0\leq \alpha\leq i$.
Thus,
\begin{eqnarray}
\triangle f_i&=&\sum_{a}\mathrm{Hess}f_i(e_a,e_a)\nonumber\\
&=&\sum_{a}
\widetilde{\mathrm{Hess}}F_i(e_a,e_a)+\sum_{a}B(e_a,e_a)(F_i)\nonumber\\
&=&\sum_{a}\widetilde{\mathrm{Hess}}F_i(e_a,e_a)+\langle\sum_{a}B(e_a,e_a),x\rangle x(F_i)\nonumber\\
&=&\sum_{a}\widetilde{\mathrm{Hess}}F_i(e_a,e_a)-2(2l-i-2)f_i,\nonumber
\end{eqnarray}
where we have used $\langle\sum_{a}B(e_a,e_a), x\rangle=-(2l-i-2)$ and $x(F_i)=2f_i$.
Moreover,
\begin{eqnarray}
\sum_{a}\widetilde{\mathrm{Hess}}F_i(e_a,e_a)&=&\tilde{\triangle}F_i-\sum_{\alpha=0}^{i}\widetilde{\mathrm{Hess}}F_i(P_{\alpha}x,P_{\alpha}x)-
\widetilde{\mathrm{Hess}}F_i(x, x)\nonumber\\
&=&\mathrm{Tr}P_{i+1}+2(i+1)f_i-2f_i\nonumber\\
&=&2if_i,\nonumber
\end{eqnarray}
where $\tilde{\triangle}F_i=\mathrm{Tr}P_{i+1}=0$ and $\widetilde{\mathrm{Hess}}F_i(P_{\alpha}x,P_{\alpha}x)=-2f_i$ by using
the properties of symmetric Clifford system. In conclusion,
$\triangle f_i=-4(l-i-1)f_i.$

Next, we will deal with the function $g_i$ on $N_i$. According to \cite{FKM81}, $N_i$ is also
a smooth submanifold in
$S^{2l-1}(1)\subset \mathbb{R}^{2l}$ with $\mathrm{dim}N_i=l+i-1$.
It is also convenient to regard the function $g_i$ as
the restriction of
the function $G_i:\mathbb{R}^{2l}\rightarrow\mathbb{R}, x\mapsto \langle P_{i}x, x\rangle$ to $N_i$.
With no possibility of confusion, we will also denote the covariant derivative and Laplacian of $N_i$ by
$\nabla$ and $\triangle$, respectively.
For any $x\in N_i$, we can define $\mathcal{P}=\sum_{\alpha=0}^{i}\langle P_{\alpha}x, x\rangle P_{\alpha}$.
And thus $\mathcal{P}x=x$. According to \cite{FKM81},
the normal space of $N_i$ in $\mathbb{R}^{2l}$ at $x$, denoted by $\nu_x N_i$, is equal to
$$\{\varsigma\in E_-(\mathcal{P})~|~\varsigma \bot Qx, \forall Q\in
\Sigma(P_0,\cdots,P_i), \langle Q, \mathcal{P}\rangle=0\}\oplus \mathbb{R}x,$$
where $E_-(\mathcal{P})$ is eigenspace of $\mathcal{P}$
for the eigenvalue $-1$ and $\Sigma(P_0,\cdots,P_i)$
the Clifford sphere spanned by $P_0,\cdots,P_i$. Since $\tilde{\nabla}G_i=2P_ix$ and $\langle Q, \mathcal{P}\rangle=0$,
$$\nabla g_i=2P_ix-2\langle P_ix, x\rangle x=2(P_i-\langle P_ix, x\rangle\mathcal{P})x=2Qx,$$
where $Q=P_i-\langle P_ix, x\rangle\mathcal{P}$ (also see \cite{TY12}).
As a result, $|\nabla g_i|^2=4(1-g_i^2).$
At last, due to \cite{So92}, the equation of Laplacian of $g_i$ holds.
\end{proof}
Furthermore, the following lemma investigates the extrinsic geometry of isoparametric hypersurfaces given by the preceding lemma.
\begin{lem}\label{principal curvature}
(1). For the isoparametric function $f_i$ on $M_i$ and any $c\in (-1,~1)$, the regular level set
$\mathcal{U}_c=f_i^{-1}(c)$ has $3$ distinct constant principal curvatures
$-\sqrt{\frac{1-c}{1+c}}$, 0, $\sqrt{\frac{1+c}{1-c}}$ with
multiplicities $l-i-2$, $i+1$, and $l-i-2$ respectively,
\emph{w.r.t.} the unit normal $\xi=\frac{\nabla f_i}{|\nabla f_i|}$.
Moreover, for any $x\in\mathcal{U}_c$, the corresponding principal spaces are
\begin{eqnarray}
T_{-\sqrt{\frac{1-c}{1+c}}}(\xi)&=&E_+(P_{i+1})\cap T_x\mathcal{U}_c,\nonumber\\
T_{\sqrt{\frac{1+c}{1-c}}}(\xi)&=&E_-(P_{i+1})\cap T_x\mathcal{U}_c,\nonumber\\
T_0(\xi)&=&\{Q\xi~|~Q\in\mathbb{R}\Sigma(P_0,\cdots,P_i)\},\nonumber
\end{eqnarray}
where $E_\pm(P_{i+1})$ are eigenspaces of $P_{i+1}$ for the eigenvalues $\pm1$ with
$\mathrm{dim}E_\pm(P_{i+1})=l$, and
$\Sigma(P_0,\cdots,P_i)$ is the Clifford sphere spanned by $P_0,\cdots,P_i$.

(2). For the isoparametric function $g_i$ on $N_i$ and any $c\in (-1,~1)$, the regular level set
$\mathcal{V}_c=g_i^{-1}(c)$ has $3$ distinct constant principal curvatures
$-\sqrt{\frac{1-c}{1+c}}$, $\sqrt{\frac{1+c}{1-c}}$, 0,  with
multiplicities $i-1$, $i-1$ and $l-i$, respectively,
\emph{w.r.t.} the unit normal $\eta=\frac{\nabla g_i}{|\nabla g_i|}$.
Moreover, for any $x\in\mathcal{V}_c$, the corresponding principal spaces are
\begin{eqnarray}
T_{-\sqrt{\frac{1-c}{1+c}}}(\eta)&=&E_+(P_{i})\cap T_x\mathcal{V}_c=\mathrm{Span}\{Q(x-P_ix)|Q\in\Sigma(P_0,\cdots,P_{i-1}), \langle Q, \mathcal{Q}\rangle=0\},\nonumber\\
T_{\sqrt{\frac{1+c}{1-c}}}(\eta)&=&E_-(P_{i})\cap T_x\mathcal{V}_c=\mathrm{Span}\{Q(x+P_ix)|Q\in\Sigma(P_0,\cdots,P_{i-1}), \langle Q, \mathcal{Q}\rangle=0\},\nonumber\\
T_0(\eta)&=&\{X\in E_+(\mathcal{P})|\langle X,x\rangle=0, \langle X, P_iRx\rangle=0, \forall R\in \Sigma(P_0,\cdots,P_i), \langle R, \mathcal{P}\rangle=0\},\nonumber
\end{eqnarray}
where $\mathcal{P}=\sum_{\alpha=0}^{i}\langle P_{\alpha}x, x\rangle P_{\alpha}$,
and $\mathcal{Q}=\frac{1}{\sqrt{1-c^2}}(\langle P_0x, x\rangle P_0+\cdots+\langle P_{i-1}x, x\rangle P_{i-1})$.
\end{lem}
\begin{proof}
(1). For any $c\in (-1,~1)$, it follows from Lemma \ref{isoparametric}
that $\mathcal{U}_c$ is an isoparametric hypersurface in $M_i$ with $\mathrm{dim}\mathcal{U}_c=2l-i-3$,
and for each $x\in \mathcal{U}_c\subset M_i$, the unit normal
$\xi=\frac{1}{\sqrt{1-c^2}}(P_{i+1}x-cx)$. Hence, the corresponding shape operator
$A_{\xi}:T_x\mathcal{U}_c\rightarrow
T_x\mathcal{U}_c$ is given by
$A_{\xi}X=-\frac{1}{\sqrt{1-c^2}}((P_{i+1}X)^T-cX)$, for each $X\in T_x\mathcal{U}_c$, where $(P_{i+1}X)^T$
is the tangential component of
$P_{i+1}X$ in $T_x\mathcal{U}_c$.

Suppose $X\in E_+(P_{i+1})\cap T_x\mathcal{U}_c$. Then $A_{\xi}X=-\sqrt{\frac{1-c}{1+c}}X$ and $X\in
T_{-\sqrt{\frac{1-c}{1+c}}}(\xi)$. Hence,
$E_+(P_{i+1})\cap T_x\mathcal{U}_c\subset T_{-\sqrt{\frac{1-c}{1+c}}}(\xi)$.
Notice that
\begin{eqnarray}
E_+(P_{i+1})\cap T_x\mathcal{U}_c&=&\{X\in E_+(P_{i+1})~|~\langle X, x\rangle=0, \langle X, \xi\rangle=0,
\langle X, Qx\rangle=0, \forall Q \in \Sigma(P_0,\cdots,P_i)\}\nonumber\\
&=&\{X\in E_+(P_{i+1})~|~\langle X, Qx\rangle=0, \forall Q \in \Sigma(P_0,\cdots,P_i, P_{i+1})\},\nonumber
\end{eqnarray}
so $\mathrm{dim}(E_+(P_{i+1})\cap T_x\mathcal{U}_c)\geq l-i-2$.

Next, suppose $X\in E_-(P_{i+1})\cap T_x\mathcal{U}_c$. Then as above, we have $A_{\xi}X=\sqrt{\frac{1+c}{1-c}}X$ and $X\in
T_{\sqrt{\frac{1+c}{1-c}}}(\xi)$. Hence,
$E_-(P_{i+1})\cap T_x\mathcal{U}_c\subset T_{\sqrt{\frac{1+c}{1-c}}}(\xi)$ and
\begin{eqnarray}
E_-(P_{i+1})\cap T_x\mathcal{U}_c&=&\{X\in E_-(P_{i+1})~|~\langle X, Qx\rangle=0, \forall Q \in \Sigma(P_0,\cdots,P_i, P_{i+1})\},\nonumber
\end{eqnarray}
which implies that $\mathrm{dim}(E_-(P_{i+1})\cap T_x\mathcal{U}_c)\geq l-i-2$.

At last, suppose $X=Q\xi$ for $Q\in\Sigma(P_0,\cdots,P_i)$. We need to show that
$X\in T_x\mathcal{U}_c$.
Observe that
\begin{eqnarray}
\langle X, x\rangle&=&\langle\xi, Qx\rangle=\frac{1}{\sqrt{1-c^2}}\langle P_{i+1}x-cx, Qx\rangle=0,\nonumber\\
\langle X, \xi\rangle&=&\langle Q\xi, \xi\rangle=\frac{1}{1-c^2}\langle Q(P_{i+1}x-cx), P_{i+1}x-cx\rangle=0,\nonumber\\
\langle X, Qx\rangle&=&\langle\xi, x\rangle=\frac{1}{\sqrt{1-c^2}}\langle P_{i+1}x-cx, x\rangle=0.\nonumber
\end{eqnarray}
Moreover, for $P\in \Sigma(P_0,\cdots,P_i)$ with $\langle P, Q\rangle=0$,
\begin{eqnarray}
\langle X, Px\rangle&=&\frac{1}{\sqrt{1-c^2}}\langle Q(P_{i+1}x-cx) Px\rangle,\nonumber\\
&=&\frac{1}{\sqrt{1-c^2}}\langle QP_{i+1}x, Px\rangle-\frac{c}{\sqrt{1-c^2}}\langle Qx, Px\rangle\nonumber\\
&=&0,\nonumber
\end{eqnarray}
where the fact that $PQP_{i+1}$ is skew-symmetric has been used. Thus $X\in T_x\mathcal{U}_c$.
Now, $A_{\xi}X=-\frac{1}{\sqrt{1-c^2}}((P_{i+1}Q\xi)^T-cQ\xi)=0,$
since
\begin{eqnarray}
(P_{i+1}Q\xi)^T&=&\frac{1}{\sqrt{1-c^2}}(P_{i+1}Q(P_{i+1}x-cx))^T\nonumber\\
&=&\frac{1}{\sqrt{1-c^2}}(-Qx+cQP_{i+1}x)^T\nonumber\\
&=&\frac{1}{\sqrt{1-c^2}}(cQP_{i+1}x-c^2Qx)^T\nonumber\\
&=&cQ\xi.\nonumber
\end{eqnarray}
So $\{Q\xi~|~Q\in\mathbb{R}\Sigma(P_0,\cdots,P_i)\}\subset T_0(\xi)$ and
$\mathrm{dim}\{Q\xi~|~Q\in\mathbb{R}\Sigma(P_0,\cdots,P_i)\}=i+1$.

We have constructed three mutually orthogonal subspaces of $T_x\mathcal{U}_c$ and the sum of the dimensions
is no less than $i+1+2(l-i-2)=\mathrm{dim} T_x\mathcal{U}_c$. Hence part (1) of the lemma follows.

(2). Analogous to part (1), for any $c\in (-1,~1)$, it
also follows from Lemma \ref{isoparametric}
that $\mathcal{V}_c$ is an isoparametric hypersurface in $N_i$ with $\mathrm{dim}\mathcal{V}_c=l+i-2$.
For each $x\in \mathcal{V}_c\subset N_i$, the unit normal
$\eta=\frac{1}{\sqrt{1-c^2}}(P_{i}x-cx)$ and the corresponding shape operator
$A_{\eta}:T_x\mathcal{V}_c\rightarrow
T_x\mathcal{V}_c$ is given by
$A_{\eta}X=-\frac{1}{\sqrt{1-c^2}}((P_{i}X)^T-cX)$, for each $X\in T_x\mathcal{V}_c$, where $(P_{i}X)^T$
is the tangential component of
$P_{i}X$ in $T_x\mathcal{V}_c$.

Suppose $X=Q(x-P_ix)$ for $Q\in\Sigma(P_0,\cdots,P_{i-1})$ with $\langle Q, \mathcal{Q}\rangle=0$.
Then
$$P_{i}X=P_iQ(x-P_ix)=Q(-P_ix+x)=X,$$
and $X\in E_+(P_i)$. The next task for us is to show
$X\in T_x\mathcal{V}_c$. It follows from $\langle Q, \mathcal{Q}\rangle=0$ that $\langle Q,
\mathcal{P}\rangle=0$, and consequently,
\begin{eqnarray}
\langle X, x\rangle&=&\langle Q(x-P_ix), x\rangle\nonumber\\
&=&\langle Qx, x\rangle-\langle Qx, P_ix\rangle\nonumber\\
&=&\langle Qx, \mathcal{P}x\rangle\nonumber\\
&=&0.\nonumber
\end{eqnarray}
and
\begin{eqnarray}
\langle X, \eta\rangle&=&\frac{1}{\sqrt{1-c^2}}\langle Q(x-P_ix), P_ix-cx\rangle\nonumber\\
&=&\sqrt{\frac{1-c}{1+c}}\langle Qx, x\rangle\nonumber\\
&=&\sqrt{\frac{1-c}{1+c}}\langle Qx, \mathcal{P}x\rangle\nonumber\\
&=&0,\nonumber
\end{eqnarray}
where the equalities $\langle Q, P_i\rangle=0$ and $\mathcal{P}x=x$ have been used.
Since
$$\nu_x N_i=\{\varsigma\in E_-(\mathcal{P})~|~\varsigma \bot Qx, \forall Q\in
\Sigma(P_0,\cdots,P_i), \langle Q, \mathcal{P}\rangle=0\}\oplus \mathbb{R}x$$
as in Lemma \ref{isoparametric}, to prove $X\in T_x\mathcal{V}_c$, it
is sufficient to show $\langle X, \zeta\rangle=0$ for each normal vector $\zeta\in \{\varsigma\in E_-(\mathcal{P})~|~\varsigma \bot Qx, \forall Q\in
\Sigma(P_0,\cdots,P_i), \langle Q, \mathcal{P}\rangle=0\}$. Actually, $\langle X, \zeta\rangle=\langle Qx, \zeta\rangle-\langle QP_{i}x, \zeta\rangle=-\langle QP_ix, \zeta\rangle$. Furthermore,
$\langle QP_ix, \zeta\rangle=\langle \mathcal{P}QP_ix, \mathcal{P}\zeta\rangle=-\langle \mathcal{P}QP_ix,\zeta\rangle=
\langle Q\mathcal{P}P_ix, \zeta\rangle=\langle Q(-P_i\mathcal{P}+2\langle P_ix, x\rangle I_{2l})x, \zeta\rangle=
-\langle QP_ix, \zeta\rangle+2\langle P_ix, x\rangle\langle Qx, \zeta\rangle=-\langle QP_ix, \zeta\rangle$,
where the identity $\mathcal{P}P_i=-P_i\mathcal{P}+2\langle P_ix, x\rangle I_{2l}$ has been used.
It follows that $\langle X, \zeta\rangle=0$ and $X\in T_x\mathcal{V}_c$. Then $A_{\eta}X=-\sqrt{\frac{1-c}{1+c}}X$
and
$$\mathrm{Span}\{Q(x-P_ix)|Q\in\Sigma(P_0,\cdots,P_{i-1}), \langle Q, \mathcal{Q}\rangle=0\}\subset E_+(P_{i})\cap T_x\mathcal{V}_c \subset T_{-\sqrt{\frac{1-c}{1+c}}}(\eta),$$
with $\mathrm{dim}\mathrm{Span}\{Q(x-P_ix)|Q\in\Sigma(P_0,\cdots,P_{i-1}), \langle Q, \mathcal{Q}\rangle=0\}=i-1.$

Similarly, suppose $X=Q(x+P_ix)$ for $Q\in\Sigma(P_0,\cdots,P_{i-1})$ with $\langle Q, \mathcal{Q}\rangle=0$. Then
$X\in E_-(P_{i})\cap T_x\mathcal{V}_c$, $A_{\eta}X=\sqrt{\frac{1+c}{1-c}}X$ and
$$\mathrm{Span}\{Q(x+P_ix)|Q\in\Sigma(P_0,\cdots,P_{i-1}), \langle Q, \mathcal{Q}\rangle=0\}\subset E_-(P_{i})\cap T_x\mathcal{V}_c \subset T_{\sqrt{\frac{1+c}{1-c}}}(\eta),$$
with $\mathrm{dim}\mathrm{Span}\{Q(x+P_ix)|Q\in\Sigma(P_0,\cdots,P_{i-1}), \langle Q, \mathcal{Q}\rangle=0\}=i-1.$

Now, suppose $X\in E_+(\mathcal{P}), \langle X,x\rangle=0$, and $\langle X, P_iRx\rangle=0$, for arbitrary $R\in \Sigma(P_0,\cdots,P_i)$ with $\langle R, \mathcal{P}\rangle=0$. In this case, to prove $X\in T_x\mathcal{V}_c$,
it is sufficient to verify $\langle X, \eta\rangle=0$. In fact, we have $\langle X, \eta\rangle=\frac{1}{\sqrt{1-c^2}}\langle X, P_ix-cx\rangle=\frac{1}{\sqrt{1-c^2}}\langle X, P_ix\rangle$
and
$$\langle X, P_ix\rangle=\langle \mathcal{P}X, \mathcal{P}P_ix\rangle=\langle X, -P_i\mathcal{P}x+2\langle P_ix, x\rangle x\rangle=-\langle X, P_ix\rangle,$$
where the identity $\mathcal{P}P_i=-P_i\mathcal{P}+2\langle P_ix, x\rangle I_{2l}$ has been used. Hence, $X\in T_x\mathcal{V}_c$. Then we will show that $A_{\eta}X=0$.
Observe that in this case
\begin{eqnarray}
A_{\eta}X=0&\Leftrightarrow&(P_iX)^T=cX\nonumber\\
&\Leftrightarrow&P_iX-cX\in \{\varsigma\in E_-(\mathcal{P})~|~\varsigma \bot Qx, \forall Q\in
\Sigma(P_0,\cdots,P_i), \langle Q, \mathcal{P}\rangle=0\}.\nonumber
\end{eqnarray}
It is sufficient to prove
$$P_iX-cX\in \{\varsigma\in E_-(\mathcal{P})~|~\varsigma \bot Qx, \forall Q\in
\Sigma(P_0,\cdots,P_i), \langle Q, \mathcal{P}\rangle=0\}.$$
First, it is not difficult to prove that
$$\mathcal{P}(P_iX-cX)=-(P_iX-cX) \Leftrightarrow \mathcal{P}X=X,$$
and therefore $P_iX-cX \in E_-(\mathcal{P})$.
Next, for $Q\in\Sigma(P_0,\cdots,P_i)$ with $\langle Q, \mathcal{P}\rangle=0$, we have
$\langle P_iX-cX, Qx\rangle=\langle X, P_iQx\rangle-c\langle X, Qx\rangle$ and
$\langle X, Qx\rangle=\langle\mathcal{P}X, \mathcal{P}Qx\rangle=-\langle X, Q\mathcal{P}x\rangle=
-\langle X, Qx\rangle$. It follows that
$\langle P_iX-cX, Qx\rangle=\langle X, P_iQx\rangle=0$ by the definition of $X$.
In a word, $A_{\eta}X=0$ and
$$\{X\in E_+(\mathcal{P})|\langle X,x\rangle=0, \langle X, P_iRx\rangle=0, \forall R\in \Sigma(P_0,\cdots,P_i), \langle R, \mathcal{P}\rangle=0\} \subset T_0(\eta).$$
Moreover, we claim that if $X\in E_+(\mathcal{P})$, and $\langle X, P_iRx\rangle=0$ for any $R\in \Sigma(P_0,\cdots,P_i)$ with $\langle R, \mathcal{P}\rangle=0$,
then $\langle X, x \rangle=0$. To prove the claim, define $a_0=\langle P_i, \mathcal{P}\rangle$. Then $|a_0|<1$, and $P_i-a_0\mathcal{P}\in \Sigma(P_0,\cdots,P_i)$
with $\langle P_i-a_0\mathcal{P}, \mathcal{P}\rangle=0$.
Hence, $\langle X, P_i(P_i-a_0\mathcal{P})x\rangle=\langle X, x\rangle-a_0\langle X, P_ix\rangle=0$.
Observing that $$\langle X, P_ix\rangle=\langle \mathcal{P}X, \mathcal{P}P_ix\rangle=\langle X, (-P_i\mathcal{P}+2\langle P_ix, x\rangle I_{2l})x\rangle=
-\langle X, P_ix\rangle+2\langle P_ix, x\rangle\langle X, x\rangle,$$
we see $\langle X, P_ix\rangle=c\langle X, x\rangle.$
And thus $\langle X, x\rangle-a_0c\langle X, x\rangle=0$. Since $|a_0|, |c|<1$, it follows $\langle X, x\rangle=0$.
Due to the claim above, it follows that
$$\mathrm{dim}\{X\in E_+(\mathcal{P})|\langle X,x\rangle=0, \langle X, P_iRx\rangle=0, \forall R\in \Sigma(P_0,\cdots,P_i), \langle R, \mathcal{P}\rangle=0\}\geq l-i.$$

Since three mutually orthogonal subspaces of $T_x\mathcal{V}_c$ are constructed and the sum of the dimensions
is no less than $l-i+2(i-1)=\mathrm{dim} T_x\mathcal{V}_c$, part (2) of the lemma follows.
\end{proof}

\begin{lem}\label{focal}
For $c=\pm 1$, $\mathcal{U}_{\pm 1}=E_{\pm}(P_{i+1})\cap S^{2l-1}(1)$, denoted by $SE_{\pm}(P_{i+1})$, and
the two focal submanifolds $\mathcal{U}_{\pm 1}$ are both
isometric to $S^{l-1}(1)$ and are totally geodesic in $M_i$. Similarly,
$\mathcal{V}_{\pm 1}=E_{\pm}(P_{i})\cap S^{2l-1}(1)$, denoted by $SE_{\pm}(P_{i})$, and $\mathcal{V}_{\pm 1}$ are both
isometric to $S^{l-1}(1)$ and are totally geodesic in $N_i$.
\end{lem}
\begin{proof}
Because $\mathcal{U}_{\pm 1}=\{x\in S^{2l-1}(1)~|~\langle P_0x, x\rangle=\cdots=
\langle P_ix, x\rangle=0, \langle P_{i+1}x, x \rangle=\pm 1\}$ by definition, it follows from
Cauchy-Schwarz inequality and the properties
of the symmetric Clifford system that
$\mathcal{U}_{\pm 1}=\{x\in S^{2l-1}(1)~|~P_{i+1}x=\pm x\}$.
Thus, the first part of the lemma is proved. And an analogous argument implies the second part
of the lemma.
\end{proof}
We are now in the position to give a

\noindent\textbf{Proof of Theorem \ref{filtration}:}
\begin{proof}
In order to complete the proof, it remains to show that $M_{i+j}$ is minimal in $M_i$, and $N_i$
is minimal in $N_{i+j}$, by putting Lemmas \ref{isoparametric}-\ref{focal} together. For $M_{i+j}\subset
M_i\subset S^{2l-1}(1)$, $M_{i+j}$ is minimal in $M_i$ indeed, since $M_{i+j}$ is minimal
in $S^{2l-1}(1)$ (c.f. \cite{Mu80}). Similarly, $N_i$ is also minimal in $N_{i+j}$. Now, the proof
of Theorem \ref{filtration} is complete.
\end{proof}
After finishing the proof of Theorem \ref{filtration}, we will continue to study the normal exponential map
of isoparametric hypersurfaces constructed above to prepare for the next section.

We first consider the isoparametric function $f_i$ on $M_i$. In this case, $M_{i+1}$ is the minimal
isoparametric hypersurface in $M_i$. At any $x\in M_{i+1}$, the unit normal $\xi(x)=P_{i+1}x$.
Define a map $\phi_t : \mathcal{U}_{0}=M_{i+1}\rightarrow M_i$ by $\phi_t(x)=\cos{t}x+\sin{t}\xi$.
In fact, it is not difficult to check that $\phi_t(x)\in M_i$, i.e.,
$\langle P_{\alpha}\phi_t(x), \phi_t(x)\rangle=0$ for any $0\leq\alpha \leq i$, and hence the map $\phi_t$ is well-defined.
Furthermore, by a direct computation, we can infer that $f_i(\phi_t(x))=\sin{2t}$ and $\phi_t(x)\in \mathcal{U}_{\sin{2t}}$. For simplicity, we denote $\phi_t(x)$ by $x_t$.

For the case of the isoparametric function $g_i$ on $N_i$, $N_{i-1}$ is the minimal isoparametric hypersurface in $N_i$. At any point $x\in N_{i-1}$, the unit normal $\eta(x)=P_ix$. Define a map $\psi_t : \mathcal{V}_{0}=N_{i-1}\rightarrow N_i$ by $\psi_t(x)=\cos{t}x+\sin{t}\eta$. And it is also not difficult to check that $\psi_t(x)\in N_i$, and $g_i(\psi_t(x))\in \mathcal{V}_{\sin{2t}}$. With no possibility of confusion, we also denote $\psi_t(x)$ by $x_t$.

The following properties of maps $\phi_t$ and $\psi_t$ will be useful later.
\begin{prop}\label{tangent map}
(1). The map $\phi_t$ is the normal exponential map of $M_{i+1}$ in $M_i$. For each $x\in M_{i+1}$,
the tangent map $(\phi_t)_*: T_xM_{i+1}\rightarrow T_{x_t}M_i$ is given by

a). For $X\in T_{-1}(\xi)$, $(\phi_t)_*(X)=(\cos{t}+\sin{t})X$;

b). For $X\in T_{1}(\xi)$, $(\phi_t)_*(X)=(\cos{t}-\sin{t})X$;

c). For $X\in T_{0}(\xi)$, i.e., $X=QP_{i+1}x$, for some $Q\in \mathrm{Span}\{P_0, \cdots, P_{i}\}$, $(\phi_t)_*(X)=Q(-\sin{t}x+\cos{t}P_{i+1}x)$. In particular, $|(\phi_t)_*(X)|^2=|X|^2$.

(2). The map $\psi_t$ is the normal exponential map of $N_{i-1}$ in $N_i$. For each $x\in N_{i-1}$,
the tangent map $(\psi_t)_*: T_xN_{i-1}\rightarrow T_{x_t}N_i$ is given by

a). For $X\in T_{-1}(\eta)$, $(\psi_t)_*(X)=(\cos{t}+\sin{t})X$;

b). For $X\in T_{1}(\eta)$, $(\psi_t)_*(X)=(\cos{t}-\sin{t})X$;

c). For $X\in T_{0}(\eta)$, $(\psi_t)_*(X)=\cos{t}X+\sin{t}P_{i}X$ and $|(\psi_t)_*(X)|^2=|X|^2$.
\end{prop}
\begin{proof}
As the proof of part (2) is similar to that of part (1), we only give the proof of part (1).
Recall that the normal exponential map of $M_{i+1} \subset M_i$ is given by
$exp: M_{i+1}\times \mathbb{R}\rightarrow M_i, exp(x, t)=exp_x(t\xi)$, which
is the restriction of the exponential map exp of $M_i$ to the normal bundle
of $M_{i+1}$. Observing that the curve $\phi_t(x)=\cos{t}x+\sin{t}\xi$\quad is a geodesic
in $S^{2l-1}(1)$ issuing from
$x$ with initial vector $\xi(x)$ and $\phi_t(x)\in M_i\subset S^{2l-1}(1)$,
we obtain that $exp(x, t)=\phi_t(x)$. Therefore, $\phi_t$ is exactly
the normal exponential map of $M_{i+1}$ in $M_i$. Next, by definition of the tangent map,
for each $X\in T_xM_{i+1}$, $(\phi_t)_*(X)=\cos{t}X+\sin{t}P_{i+1}X$. Then part (1) of Proposition 2.1 follows
immediately.
\end{proof}
Based on Proposition \ref{tangent map}, we have the following remark.
\begin{rem}
If $\sin{2t}\neq \pm 1$, the maps $\phi_t$ and $\psi_t$ are essentially diffeomorphisms from
$M_{i+1}$ to the parallel hypersurface $\mathcal{U}_{\sin{2t}}$ and from $N_{i-1}$ to the parallel
hypersurface $\mathcal{V}_{\sin{2t}}$, respectively.
If $\sin{2t}= \pm 1$, the maps $\phi_t$ and $\psi_t$ (focal maps) are essentially submersions from $M_{i+1}$ to $\mathcal{U}_{\pm 1}$
and from $N_{i-1}$ to $\mathcal{V}_{\pm 1}$, respectively. However, due to Proposition \ref{tangent map},
they are not Riemannian submersions.
\end{rem}
We will conclude this section by the following result, which gives the geometry properties of the fibers
of the submersions mentioned in Remark 2.1.

\begin{prop}\label{fiber}
(1). For $t=\pm \frac{\pi}{4}$, the maps $\phi_{\pm \frac{\pi}{4}}:M_{i+1}\rightarrow\mathcal{U}_{\pm 1}$ are given by
$\phi_{\pm \frac{\pi}{4}}(x)=\frac{1}{\sqrt{2}}(x\pm P_{i+1}x)$ for $x\in M_{i+1}$. For any $y\in \mathcal{U}_1$,
the fiber $F_y=\phi_{\frac{\pi}{4}}^{-1}(y)$
is a totally geodesic submanifold in $M_{i+1}$, and is isometric to $S^{l-i-2}(\frac{1}{\sqrt{2}})$.
For any $y'\in \mathcal{U}_{-1}$, the fiber $F'_{y'}=\phi_{-\frac{\pi}{4}}^{-1}(y')$ is also a totally geodesic
submanifold in $M_{i+1}$, and is isometric to $S^{l-i-2}(\frac{1}{\sqrt{2}})$.

(2). For $t=\pm \frac{\pi}{4}$, the maps $\psi_{\pm \frac{\pi}{4}}:N_{i-1}\rightarrow\mathcal{V}_{\pm 1}$ are given by
$\psi_{\pm \frac{\pi}{4}}(x)=\frac{1}{\sqrt{2}}(x\pm P_{i}x)$ for $x\in N_{i-1}$. For any $y\in \mathcal{V}_1$,
the fiber $F_y=\psi_{\frac{\pi}{4}}^{-1}(y)$
is a totally geodesic submanifold in $N_{i-1}$, and is isometric to $S^{i-1}(\frac{1}{\sqrt{2}})$. For any $y'\in \mathcal{V}_{-1}$, the fiber $F'_{y'}=\psi_{-\frac{\pi}{4}}^{-1}(y')$ is also a totally geodesic
submanifold in $N_{i-1}$, and is isometric to $S^{i-1}(\frac{1}{\sqrt{2}})$.
\end{prop}
\begin{proof}
(1). Given $y\in\mathcal{U}_1$, it is straightforward to verify
$$F_y=\phi_{\frac{\pi}{4}}^{-1}(y)=\{\frac{1}{\sqrt{2}}(y+z)~|~z\in SE_-(P_{i+1}), \langle z, P_0y\rangle=
\cdots=\langle z, P_iy\rangle=0\}.$$
Consequently, it is not difficult to see that $F_y$ is isometric to $S^{l-i-2}(\frac{1}{\sqrt{2}})$.
Hence, the left task for us is to show that $F_y$ is totally geodesic in $M_{i+1}$.
Denote the connections of $M_{i+1}$, $S^{2l-1}(1)$ and $\mathbb{R}^{2l}$ respectively by $\nabla$, $\bar{\nabla}$ and
$\tilde{\nabla}$. For each $x\in F_y$, $x=\frac{1}{\sqrt{2}}(y+z)$ for some $z\in SE_-(P_{i+1})$ with
$\langle z, P_0y\rangle=\cdots=\langle z, P_iy\rangle=0$. Since $P_{i+1}y=y$, it is clear that
$P_0y, \cdots, P_{i}y\in E_-(P_{i+1})$. Choose $v\in T_x(F_y)$, then $v\in E_-(P_{i+1})$, $\langle v, z\rangle=0$ and
$\langle v, P_0y\rangle=\cdots=\langle v, P_iy\rangle=0$.
Define $c(t)=\frac{1}{\sqrt{2}}(y+\cos{t}z+\sin{t}v)$. Then $c(t)$ is a geodesic in $F_y$ with $c(0)=x$ and $c'(0)=\frac{1}{\sqrt{2}}v$.
By a direct computation,
\begin{eqnarray}
\bar{\nabla}_{c'(t)}c'(t)&=&\tilde{\nabla}_{c'(t)}c'(t)-\langle\tilde{\nabla}_{c'(t)}c'(t), c(t)\rangle c(t)\nonumber\\
&=&-\frac{1}{\sqrt{2}}(\cos{t}z+\sin{t}v)-\langle\tilde{\nabla}_{c'(t)}c'(t), c(t)\rangle c(t)\nonumber\\
&=&\frac{1}{2\sqrt{2}}(y-\cos{t}z-\sin{t}v).\nonumber
\end{eqnarray}
Since $\frac{1}{\sqrt{2}}(y-\cos{t}z-\sin{t}v)=P_{i+1}c(t)$ is a normal vector to $M_{i+1}$,
we have $\nabla_{c'(t)}c'(t)=0$. And it means that $c(t)$ is a geodesic in $M_{i+1}$. Therefore,
it follows that $F_y$ is totally geodesic in $M_{i+1}$.

Similarly, for $y'\in\mathcal{U}_{-1}$, it follows
$$F'_{y'}=\phi_{-\frac{\pi}{4}}^{-1}(y')=\{\frac{1}{\sqrt{2}}(y'+z)~|~z\in SE_+(P_{i+1}), \langle z, P_0y'\rangle=
\cdots=\langle z, P_iy'\rangle=0\},$$
and $F'_{y'}$ is totally geodesic in $M_{i+1}$, but the detailed proof is omitted here.

(2). For $y\in \mathcal{V}_{1}$, the fiber is given by
$$F_y=\psi_{\frac{\pi}{4}}^{-1}(y)=\{\frac{1}{\sqrt{2}}(y+z)~|~z\in \mathrm{Span}\{P_0y, \cdots, P_{i-1}y\}\},$$
and thus $F_y$ is isometric to $S^{i-1}(\frac{1}{\sqrt{2}})$. Next, we will show $F_y$ is totally geodesic in $N_{i-1}$.
Choose any $x\in F_y$ and $v\in T_{x}F_y$, then $x=\frac{1}{\sqrt{2}}(y+z)$ for some $z\in\mathrm{Span}\{P_0y,
\cdots, P_{i-1}y\}$, and $v\in \mathrm{Span}\{P_0y, \cdots, P_{i-1}y\}$ with $\langle v, z\rangle=0$.
Now, we can define $c(t)=\frac{1}{\sqrt{2}}(y+\cos{t}z+\sin{t}v)$. Clearly, $c(t)$ is a geodesic in $F_y$ with
$c(0)=x$ and $c'(0)=\frac{1}{\sqrt{2}}v$. It follows that
\begin{eqnarray}
\bar{\nabla}_{c'(t)}c'(t)&=&\tilde{\nabla}_{c'(t)}c'(t)-\langle\tilde{\nabla}_{c'(t)}c'(t), c(t)\rangle c(t)\nonumber\\
&=&-\frac{1}{\sqrt{2}}(\cos{t}z+\sin{t}v)-\langle\tilde{\nabla}_{c'(t)}c'(t), c(t)\rangle c(t)\nonumber\\
&=&\frac{1}{2\sqrt{2}}(y-\cos{t}z-\sin{t}v),\nonumber
\end{eqnarray}
where the connections of $N_{i-1}$, $S^{2l-1}(1)$ and $\mathbb{R}^{2l}$ are denoted by $\nabla$, $\bar{\nabla}$ and
$\tilde{\nabla}$, respectively. And then $\nabla_{c'(t)}c'(t)=0$, because $\frac{1}{\sqrt{2}}(y-\cos{t}z-\sin{t}v)=P_{i}c(t)$
is a normal vector to $N_{i-1}$. That is to say, $c(t)$ is a geodesic in $N_{i-1}$. Hence, $F_y$ is totally
geodesic in $N_{i-1}$.

For the case $t=-\frac{\pi}{4}$, the proof is analogous to the above and we will not go into the details.
In fact, for $y'\in\mathcal{V}_{-1}$, the fiber is given by
$$F'_{y'}=\psi_{-\frac{\pi}{4}}^{-1}(y')=\{\frac{1}{\sqrt{2}}(y'+z)~|~z\in \mathrm{Span}\{P_0y', \cdots, P_{i-1}y'\}\},$$
and $F'_{y'}$ is totally geodesic in $N_{i-1}$.
\end{proof}
\begin{rem}
Corollary \ref{fibration} follows from Theorem \ref{filtration} and Proposition \ref{fiber}.
\end{rem}

\section{Eigenvalue estimates}
Based on the isoparametric foliations constructed in Theorem \ref{filtration}, we intend to
prove Theorem \ref{eigenvalue} on eigenvalues estimate of the Laplacian in this section. We first
recall a crucial theorem
which has been used in \cite{Mu88}, \cite{TY13} and \cite{TXY14}.

\noindent \textbf{Theorem (Chavel and Feldman \cite{CF78}, Ozawa
\cite{Oz81})}\,\, {\itshape Let $V$ be a closed, connected smooth
Riemannian manifold and $W$ a closed submanifold of $V$. For any
sufficiently small $\varepsilon>0$, set $W(\varepsilon)=\{x\in V:
dist(x, W)<\varepsilon\}$. Let $\lambda^D_k(\varepsilon)$ $(k=1, 2,
\cdots)$ be the $k$-th eigenvalue of the Laplace-Beltrami operator on
$V-W(\varepsilon)$ under the Dirichlet boundary condition. If $\dim
V\geq \dim W+ 2$, then for any $k =0, 1, \cdots$
\begin{equation*}\label{lambda k limit}
\lim_{\varepsilon\to 0}\lambda^D_{k+1}(\varepsilon)=\lambda_{k}(V).
\end{equation*}}

It is necessary to point out that the
proof of part (2) of Theorem \ref{eigenvalue} is analogous to that of part (1), so
the detailed proof will be only given for part (1).

\noindent\textbf{Proof of Theorem \ref{eigenvalue} (1):}
\begin{proof}
a). Consider the isoparametric foliation on $M_i$ given by the function $f_i$, provided that
$0\leq i\leq m-1$ and $l-i-3>0$.
For sufficiently small $\varepsilon>0$, set
$$\mathcal{U}(\varepsilon)=\bigcup_{t\in[-\frac{\pi}{4}+\varepsilon, \frac{\pi}{4}-\varepsilon]}
\mathcal{U}_{\sin{2t}}.$$
Actually, $\mathcal{U}(\varepsilon)$ is a domain of $M_i$ obtained by excluding $\varepsilon$-neighborhoods
of $\mathcal{U}_1$ and $\mathcal{U}_{-1}$.
Thus, by the theorem of Chavel-Feldman and Ozawa,
$$\lim_{\varepsilon\to 0}\lambda^D_{k+1}(\mathcal{U}(\varepsilon))=\lambda_{k}(M_i).$$
Next we will estimate $\lambda^D_{k+1}(\mathcal{U}(\varepsilon))$ from above in terms of $\lambda_{k}(M_{i+1})$
by making use of the mini-max principle.

According to part (1) of Proposition \ref{tangent map}, the volume element of $\mathcal{U}(\varepsilon)$
can be expressed by the volume element of $M_{i+1}$ as
$$d\mathcal{U}(\varepsilon)=(\cos{t}+\sin{t})^{l-i-2}(\cos{t}-\sin{t})^{l-i-2}dtdM_{i+1}=(\cos{2t})^{l-i-2}dtdM_{i+1}.$$

Let $h$ be a nonnegative, increasing smooth function on $[0, \infty)$ satisfying $h=1$ on $[2, \infty)$ and
$h=0$ on $[0, 1]$. For sufficiently small $\varrho>0$, define a nonnegative smooth function $\Psi_{\varrho}$
on $[-\frac{\pi}{2}, \frac{\pi}{2}]$ by

(i) $\Psi_{\varrho}(x)=1$ on $[-\frac{\pi}{2}+2\varrho, \frac{\pi}{2}-2\varrho]$,

(ii) $\Psi_{\varrho}(x)=h(\frac{\frac{\pi}{2}-x}{\varrho})$ on $[\frac{\pi}{2}-2\varrho, \frac{\pi}{2}]$,

(iii) $\Psi_{\varrho}$ is symmetric with respect to $x=0$.

\noindent Then $|\Psi_{\varrho}'(x)|\leq \frac{1}{\varrho}C$ for $x\in [-\frac{\pi}{2}, \frac{\pi}{2}]$,
where $C=\sup~\{h'(x)| x\in [0, \infty)\}$.

Let $\varphi_k$ be the $k$-th eigenfunctions on $M_{i+1}$ which are orthogonal to each other with respect to
the square integral inner product on $M_{i+1}$ and $L_{k+1}=\mathrm{Span}\{\varphi_0, \varphi_1, \cdots, \varphi_k\}$.
For each fixed $t\in [-\frac{\pi}{4}+\varepsilon, \frac{\pi}{4}-\varepsilon]$, denote $\pi=\pi_t=\phi_t^{-1}: \mathcal{U}_{\sin2t}\rightarrow M_{i+1}$.
Given any $\varphi\in L_{k+1}$, we can define
a function $\Phi_{\varepsilon}$ on $\mathcal{U}(\varepsilon)$ by
$\Phi_{\varepsilon}(x)=\Psi_{2\varepsilon}(2t)(\varphi\circ \pi)(x)$, where $t$ is determined by $x\in \mathcal{U}_{\sin{2t}}$ and
$t\in [-\frac{\pi}{4}+\varepsilon, \frac{\pi}{4}-\varepsilon]$.
It is clear that $\Phi_{\varepsilon}$ is a smooth function on $\mathcal{U}(\varepsilon)$ satisfying
the Dirichlet boundary condition.

By the mini-max principle, we can infer that
$$\lambda^D_{k+1}(\mathcal{U}(\varepsilon))\leq \sup_{\varphi\in L_{k+1}}\frac{||\nabla\Phi_{\varepsilon}||^2_2}{||\Phi_{\varepsilon}||^2_2}.$$
Now, we will estimate the term $\frac{||\nabla\Phi_{\varepsilon}||^2_2}{||\Phi_{\varepsilon}||^2_2}$.
Since the normal geodesic starting from $M_{i+1}$ is perpendicular to any parallel hypersurface $\mathcal{U}_c$,
it follows that
$$||\nabla \Phi_{\varepsilon}||^2_2=\int_{\mathcal{U}(\varepsilon)}4(\Psi_{2\varepsilon}'(2t))^2(\varphi\circ\pi)^2
d\mathcal{U}(\varepsilon)+\int_{\mathcal{U}(\varepsilon)}(\Psi_{2\varepsilon}(2t))^2|\nabla(\varphi\circ\pi)|^2
d\mathcal{U}(\varepsilon).$$
Moreover,
\begin{eqnarray}
||\Phi_{\varepsilon}||^2_2&=&\int_{\mathcal{U}(\varepsilon)}\Psi_{2\varepsilon}^2(2t)(\varphi\circ\pi)^2
d\mathcal{U}(\varepsilon)\nonumber\\
&=&\int_{-\frac{\pi}{4}+\varepsilon}^{\frac{\pi}{4}+\epsilon}\int_{M_{i+1}}\Psi_{2\varepsilon}^2(2t)
(\cos{2t})^{l-i-2}\varphi^2dtdM_{i+1}\nonumber\\
&=&\frac{||\varphi||^2_2}{2}\int_{-\frac{\pi}{2}+2\varepsilon}^{\frac{\pi}{2}-2\varepsilon}
\Psi_{2\varepsilon}^2(\tau)(\cos{\tau})^{l-i-2}d\tau,\nonumber
\end{eqnarray}
and thus,
$$\frac{||\nabla\Phi_{\varepsilon}||^2_2}{||\Phi_{\varepsilon}||^2_2}=I(\varepsilon)+II(\varepsilon),$$
where
\begin{eqnarray}
I(\varepsilon)&=&\frac{\int_{\mathcal{U}(\varepsilon)}4(\Psi_{2\varepsilon}'(2t))^2(\varphi\circ\pi)^2
d\mathcal{U}(\varepsilon)}{\int_{\mathcal{U}(\varepsilon)}\Psi_{2\varepsilon}^2(2t)(\varphi\circ\pi)^2
d\mathcal{U}(\varepsilon)}\nonumber\\
&=&\frac{4\int_{-\frac{\pi}{2}+2\varepsilon}^{\frac{\pi}{2}-2\varepsilon}(\Psi_{2\varepsilon}'(\tau))^2
(\cos{\tau})^{l-i-2}d\tau}{\int_{-\frac{\pi}{2}+2\varepsilon}^{\frac{\pi}{2}-2\varepsilon}
\Psi_{2\varepsilon}^2(\tau)(\cos{\tau})^{l-i-2}d\tau},\nonumber
\end{eqnarray}
and
$$II(\varepsilon)=\frac{\int_{\mathcal{U}(\varepsilon)}(\Psi_{2\varepsilon}(2t))^2|\nabla(\varphi\circ\pi)|^2
d\mathcal{U}(\varepsilon)}{\int_{\mathcal{U}(\varepsilon)}\Psi_{2\varepsilon}^2(2t)(\varphi\circ\pi)^2
d\mathcal{U}(\varepsilon)}.$$
Observing that if $l-i-3>0$,
\begin{eqnarray}
&&\int_{-\frac{\pi}{2}+2\varepsilon}^{\frac{\pi}{2}-2\varepsilon}(\Psi_{2\varepsilon}'(\tau))^2
(\cos{\tau})^{l-i-2}d\tau\nonumber\\
&\leq&\int_{-\frac{\pi}{2}+2\varepsilon}^{-\frac{\pi}{2}+4\varepsilon}
(\Psi_{2\varepsilon}'(\tau))^2
(\cos{\tau})^{l-i-2}d\tau+\int_{\frac{\pi}{2}-4\varepsilon}^{\frac{\pi}{2}-2\varepsilon}
(\Psi_{2\varepsilon}'(\tau))^2
(\cos{\tau})^{l-i-2}d\tau\nonumber\\
&\leq& \int_{-\frac{\pi}{2}+2\varepsilon}^{-\frac{\pi}{2}+4\varepsilon}
\frac{C^2}{4{\varepsilon}^2}\cos^2{\tau}d\tau+\int_{\frac{\pi}{2}-4\varepsilon}^{\frac{\pi}{2}-2\varepsilon}
\frac{C^2}{4{\varepsilon}^2}\cos^2{\tau}d\tau,\nonumber
\end{eqnarray}
we deduce $\lim_{\varepsilon\rightarrow 0}I(\varepsilon)=0$.

It remains to consider the term $II(\varepsilon)$.
Decompose
$$\nabla\varphi=Z_1+Z_2+Z_3\in T_{-1}(\xi)\oplus T_1(\xi) \oplus T_0(\xi)=TM_{i+1}.$$
By definition, $\langle \nabla(\varphi\circ\pi), X\rangle=\langle \nabla\varphi,
\pi_{*}X\rangle$ for $X\in T\mathcal{U}_{\sin{2t}}$. From Proposition \ref{tangent map},
$$|\nabla(\varphi\circ\pi)|^2=\frac{1}{\kappa_1^2}|Z_1|^2+ \frac{1}{\kappa_2^2}|Z_2|^2+
\frac{1}{\kappa_3^2}|Z_3|^2,$$
where $\kappa_1=\cos{t}+\sin{t}$, $\kappa_1=\cos{t}-\sin{t}$, and $\kappa_3=1$.
Define
$$K_1=\int_{-\frac{\pi}{4}}^{\frac{\pi}{4}}\frac{(\cos{2t})^{l-i-2}}{\kappa_1^2}dt,
K_2=\int_{-\frac{\pi}{4}}^{\frac{\pi}{4}}\frac{(\cos{2t})^{l-i-2}}{\kappa_2^2}dt,
K_3=G=\int_{-\frac{\pi}{4}}^{\frac{\pi}{4}}(\cos{2t})^{l-i-2}dt,$$
and $K=\max\{K_1, K_2, K_3\}$.
Then
$$\lim_{\varepsilon\rightarrow 0}II(\varepsilon)=\frac{\sum_{\alpha=1}^3K_{\alpha}||Z_{\alpha}||^2_2}
{||\varphi||^2_2G}\leq \frac{K}{G}\frac{||\nabla\varphi||^2_2}{||\varphi||^2_2}.$$
Furthermore,
$$\lambda_{k}(M_i)=\lim_{\varepsilon\rightarrow 0}\lambda^D_{k+1}(\mathcal{U}(\varepsilon))\leq \lim_{\varepsilon\rightarrow 0}\sup_{\varphi\in L_{k+1}}\frac{||\nabla\Phi_{\varepsilon}||^2_2}{||\Phi_{\varepsilon}||^2_2}\leq\frac{K}{G}\lambda_{k}(M_{i+1}).$$
A direct computation yields
$$\frac{K}{G}=\frac{K_1}{G}=\frac{l-i-2}{l-i-3},$$
and the inequality a) of Theorem \ref{eigenvalue} (1) follows.

b). According to Proposition \ref{tangent map} and \ref{fiber}, the map $\phi_{\frac{\pi}{4}}: M_{i+1}\rightarrow\mathcal{U}_{1}$ is a
smooth submersion (but not a Riemannian submersion), and for any $y\in \mathcal{U}_1$,
the fiber $F_y=\phi_{\frac{\pi}{4}}^{-1}(y)$, isometric to $S^{l-i-2}(\frac{1}{\sqrt{2}})$,
is a totally geodesic submanifold in $M_{i+1}$. Moreover, For each $y\in\mathcal{U}_1$, at
a point $x\in \phi_{\frac{\pi}{4}}^{-1}(y)$, we have a decomposition $T_xM_{i+1}=T_{-1}(\xi)\oplus T_{1}(\xi)\oplus T_{0}(\xi)$, and

$(\phi_{\frac{\pi}{4}})_*(X)=(\cos{\frac{\pi}{4}}+\sin{\frac{\pi}{4}})X=\sqrt{2}X$, for $X\in T_{-1}(\xi)$;

$(\phi_{\frac{\pi}{4}})_*(X)=(\cos{\frac{\pi}{4}}-\sin{\frac{\pi}{4}})X=0$, for $X\in T_{1}(\xi)$;

$(\phi_{\frac{\pi}{4}})_*(X)=Q(-\sin{\frac{\pi}{4}}x+\cos{\frac{\pi}{4}}P_{i+1}x)$, for $X\in T_{0}(\xi)$, i.e.,
$X=QP_{i+1}x$ with $Q\in \mathrm{Span}\{P_0, \cdots, P_{i}\}$. In particular, $|(\phi_{\frac{\pi}{4}})_*(X)|^2=|X|^2$.

Using these facts, we will show
the inequality b) in Theorem \ref{eigenvalue} (1) as follows. Let $\varphi_k$ be the $k$-th
eigenfunctions on $\mathcal{U}_{1}$ which are orthogonal to each other with respect to
the square integral inner product on $\mathcal{U}_{1}$ and $L_{k+1}=\mathrm{Span}\{\varphi_0, \varphi_1, \cdots, \varphi_k\}$. For any $\varphi\in L_{k+1}$, define a function $\Phi$ on $M_{i+1}$ by
$\Phi(x)=(\varphi\circ\phi_{\frac{\pi}{4}})(x).$ By the min-max principle again, we get
$$\lambda_k(M_{i+1})\leq\sup_{\varphi\in L_{k+1}}\frac{||\nabla \Phi||^2_2}{||\Phi||^2_2}=\sup_{\varphi\in L_{k+1}}\frac{\int_{M_{i+1}}|\nabla\Phi|^2dM_{i+1}}{\int_{M_{i+1}}\Phi^2dM_{i+1}}.$$
Hence, the term $|\nabla\Phi|^2$ has to be estimated. In fact, by the properties of
$(\phi_{\frac{\pi}{4}})_*$ described above, it follows that $|\nabla\Phi|^2_x\leq 2|\nabla\varphi|^2_y$.
Then
\begin{eqnarray}
\int_{M_{i+1}}|\nabla\Phi|^2dM_{i+1}&=&
\int_{M_{i+1}}\frac{|\nabla\Phi|^2}{(\sqrt{2})^{l-i-2}}(\phi_{\frac{\pi}{4}})^*(d\mathcal{U}_1)
dS^{l-i-2}(\frac{1}{\sqrt{2}})\nonumber\\
&\leq&\int_{M_{i+1}}\frac{2|\nabla\varphi|^2}{(\sqrt{2})^{l-i-2}}(\phi_{\frac{\pi}{4}})^*(d\mathcal{U}_1)
dS^{l-i-2}(\frac{1}{\sqrt{2}})\nonumber\\
&=&\frac{2}{(\sqrt{2})^{l-i-2}}Vol(S^{l-i-2}(\frac{1}{\sqrt{2}}))\int_{\mathcal{U}_1}
|\nabla\varphi|^2d\mathcal{U}_1,\nonumber
\end{eqnarray}
and
\begin{eqnarray}
\int_{M_{i+1}}\Phi^2dM_{i+1}&=&
\int_{M_{i+1}}\frac{\Phi^2}{(\sqrt{2})^{l-i-2}}(\phi_{\frac{\pi}{4}})^*(d\mathcal{U}_1)
dS^{l-i-2}(\frac{1}{\sqrt{2}})\nonumber\\
&=&\int_{M_{i+1}}\frac{\varphi^2}{(\sqrt{2})^{l-i-2}}(\phi_{\frac{\pi}{4}})^*(d\mathcal{U}_1)
dS^{l-i-2}(\frac{1}{\sqrt{2}})\nonumber\\
&=&\frac{1}{(\sqrt{2})^{l-i-2}}Vol(S^{l-i-2}(\frac{1}{\sqrt{2}}))\int_{\mathcal{U}_1}
\varphi^2d\mathcal{U}_1.\nonumber
\end{eqnarray}
Therefore,
$$\lambda_k(M_{i+1})\leq\sup_{\varphi\in L_{k+1}}\frac{||\nabla \Phi||^2_2}{||\Phi||^2_2}\leq
2\sup_{\varphi\in L_{k+1}}\frac{||\nabla \varphi||^2_2}{||\varphi||^2_2}=2\lambda_{k}(\mathcal{U}_1)=2\lambda_{k}(S^{l-1}(1))$$
as required.
\end{proof}

\begin{rem}
(1). According to \cite{TY13}, $\lambda_k(M_i)\geq\frac{l-i-2}{l-1}\lambda_k(S^{2l-1}(1))$,
for $i\geq2$. Combining the inequality a) of Theorem \ref{eigenvalue} (1) with the inequality
of \cite{TY13} for the $i$ case, we can infer that
$\lambda_k(M_{i+1})\geq\frac{l-i-3}{l-1}\lambda_k(S^{2l-1}(1))$, which is the inequality
of \cite{TY13} for the $i+1$ case.

(2). For the isoparametric foliation on $M_i$ determined by $f_i$, as in Section 3 of
\cite{TY13}, we can also obtain the inequality $\lambda_{k}(M_i)\leq \frac{2(l-i-2)}{l-i-3}
\lambda_{k}(S^{l-1}(1))$, which is a consequence of the inequalities a) and b) of
Theorem \ref{eigenvalue} (1).

(3). Using the method of the proof for the inequality b) of Theorem \ref{eigenvalue} (1),
for a minimal isoparametric hypersurface $M$ in the unit sphere with $g=4$, multiplicities
$(m_1, m_2)$ and focal submanifolds $M_+$ and $M_-$ of codimension $m_1+1$ and $m_2+1$ respectively,
we can show that $\lambda_k(M)\leq \frac{m_1+m_2}{m_2}\lambda_k(M_+)$ and
$\lambda_k(M)\leq \frac{m_1+m_2}{m_1}\lambda_k(M_-)$.

(4). The remarks above on the isoparametric function $f_i$ on $M_i$ are
also available for the case of the isoparametric function $g_i$ on $N_i$.
\end{rem}

Next, let us focus on eigenvalue estimates in a specific case.
As is well known, for $g=4$, $(m_1,m_2)=(4,3)$, there are exactly two non-congruent families (one is homogenous
and the other is not) of isoparametric hypersurfaces of OT-FKM type. For the
homogeneous case, Tang, Xie and Yan \cite{TXY14} determined the first eigenvalue
of the focal submanifold $M_+^{10}$, that is, $\lambda_1(M_+^{10})=10$. However,
for the inhomogeneous case, the corresponding work is still open. To study the spectrum of the focal
submanifold in this case, we establish the following result.

\begin{prop}
Let $\{P_0,\cdots,P_4\}$ on $\mathbb{R}^{16}$ be a symmetric Clifford system.

1). For the case $P_0\cdots P_4=\pm I_{16}$, the corresponding isoparametric foliation is
homogenous and the focal submanifold $M_+^{10}$ is isometric to $\mathrm{Sp}(2)$ with certain
bi-invariant metric. In particular, $\lambda_{17}(M_+^{10})=16$.

2). For the case $P_0\cdots P_4\neq\pm I_{16}$, the corresponding isoparametric foliation is
inhomogeneous and the focal submanifold $\tilde{M}_+^{10}$ is only diffeomorphic
to $S^3\times S^7$, but not isometric to the product of two round spheres.
Moreover, $\lambda_{17}(\tilde{M}_+^{10})\leq 12$.
\end{prop}
\begin{proof}
1). Up to orthogonal transformations, the symmetric Clifford system in this case can
be chosen as follows. First, using the multiplication of quaternions, we can
define three orthogonal transformations $E_1, E_2,
E_3$ on $\mathbb{R}^8=\mathbb{H}\oplus\mathbb{H}$, where any point in $\mathbb{R}^8$
is considered as two quaternions. For $u=(u_1, u_2)\in \mathbb{H}^2$,
\begin{eqnarray}
E_1(u)&=&(\mi u_1,\mi u_2),\nonumber\\
E_2(u)&=&(\mj u_1,\mj u_2),\nonumber\\
E_3(u)&=&(\mk u_1,\mk u_2).\nonumber
\end{eqnarray}
Furthermore, by identifying $\mathbb{R}^{16}$ with $\mathbb{H}^4$, we can define
\begin{eqnarray}
P_0(u, v)&=&(u, -v),\nonumber\\
P_1(u, v)&=&(v, u),\nonumber\\
P_2(u, v)&=&(E_1v, -E_1u),\nonumber\\
P_3(u, v)&=&(E_2v, -E_2u),\nonumber\\
P_4(u, v)&=&(E_3v, -E_3u),\nonumber
\end{eqnarray}
for $(u,v)=(u_1, u_2; v_1, v_2)\in \mathbb{H}^4=\mathbb{R}^{16}$.
Clearly, $P_0P_1\cdots P_4=\mathrm{Id}$. According to \cite{FKM81}, the corresponding
isoparametric foliation is homogeneous, with one of the focal submanifolds
$M^{10}_+=\{x\in S^{15}(1)|\langle
P_{\alpha}x, x\rangle=0, 0\leq\alpha\leq 4\}$.
It is not difficult to show $x=(u, v)=(u_1, u_2; v_1, v_2)\in M^{10}_+$ if and only if
$$|u|=|v|=\frac{1}{\sqrt{2}},\quad u_1\overline{v_1}+u_2\overline{v_2}=0.$$
Thus we can define a map $G: M_+^{10}\rightarrow \mathrm{Sp}(2)$ by
$G(u_1, u_2; v_1, v_2)=\sqrt{2}
\left(  \begin{array}{cc}
u_1 & u_2  \\
v_1 & v_2  \\
\end{array}\right)$. It is evident that $G$ is a diffeomorphism. Furthermore, it is an isometry if $\mathrm{Sp}(2)$ is equipped
with the bi-invariant metric normalized such that the tangent vector
$\left( \begin{array}{cc}
\sqrt{2}\mi & 0  \\
0 & 0  \\
\end{array}\right)\in T_{I}\mathrm{Sp}(2)$ has unit length. For this metric, the spectrum of $\mathrm{Sp}(2)$ can be determined completely (c.f.
\cite{BM77} and \cite{Fe80}). Particularly, $\lambda_{17}(M_+^{10})=16$.

2). Also using the multiplication of quaternions, we can
define another three orthogonal transformations $E_1, E_2,
E_3$ on $\mathbb{R}^8=\mathbb{H}\oplus\mathbb{H}$. For $u=(u_1, u_2)\in \mathbb{H}^2$,
\begin{eqnarray}
E_1(u)&=&(\mi u_1,-\mi u_2),\nonumber\\
E_2(u)&=&(\mj u_1,-\mj u_2),\nonumber\\
E_3(u)&=&(\mk u_1,-\mk u_2).\nonumber
\end{eqnarray}
And similarly define
\begin{eqnarray}
P_0(u, v)&=&(u, -v),\nonumber\\
P_1(u, v)&=&(v, u),\nonumber\\
P_2(u, v)&=&(E_1v, -E_1u),\nonumber\\
P_3(u, v)&=&(E_2v, -E_2u),\nonumber\\
P_4(u, v)&=&(E_3v, -E_3u),\nonumber
\end{eqnarray}
for $(u,v)=(u_1, u_2; v_1, v_2)\in \mathbb{H}^4=\mathbb{R}^{16}$.
In this case, $P_0P_1\cdots P_4\neq\pm\mathrm{Id}$. According to \cite{FKM81},
the corresponding isoparametric foliation is inhomogeneous. Now,
$x=(u, v)=(u_1, u_2; v_1, v_2)\in\tilde{M}^{10}_+$ (one of the focal submanifolds) if and only if
$$|u|=|v|=\frac{1}{\sqrt{2}}, \quad v_1\overline{u_1}+u_2\overline{v_2}=0.$$
Moreover, an explicit diffeomorphism
$F$ from $\tilde{M}^{10}_+$ to $S^{7}(\frac{1}{\sqrt{2}})\times S^3(1)$
can be constructed by
$$F(u_1, u_2; v_1, v_2)=(u_1, u_2; 2(\overline{u_2}v_1-\overline{v_2}u_1)).$$
Observe that $F$ is not an isometry from $\tilde{M}^{10}_+$ to $S^{7}(\frac{1}{\sqrt{2}})\times S^3(1)$
with the standard product metric.
In light of the diffeomorphism $F$, it is not difficult to show the four coordinate components of
$S^3(1)$, the second factor of $S^{7}(\frac{1}{\sqrt{2}})\times S^3(1)$, provide $4$ eigenfunctions on
$\tilde{ M}^{10}_+$ with the same eigenvalue
$12$. For instance, given $\Phi :\mathbb{R}^{16}\rightarrow\mathbb{R}$ by
$\Phi(u_1,u_2;v_1,v_2)=\langle u_1, v_2\rangle-\langle u_2, v_1\rangle$,
and $\varphi :\tilde{M}_+^{10}\rightarrow\mathbb{R}$ by $\varphi:=\Phi|_{\tilde{M}_+^{10}}$, then
$\triangle \varphi=-12\varphi$ by a direct computation.
On the other hand,
since $\tilde{ M}^{10}_+$ is minimal in $S^{15}(1)$, the $\mathbb{R}^{16}$-components
give 16 eigenfunctions with the same eigenvalue 10. These arguments imply that
$\lambda_{17}(\tilde{M}_+^{10})\leq 12$ as required.
\end{proof}
\begin{rem}
From the view point of representation theory, it is worth mentioning that
the symmetric Clifford system $\{P_0,\cdots,P_4\}$ on $\mathbb{R}^{16}$ with
$P_0\cdots P_4=\pm I_{16}$ cannot be extended. However, for the symmetric Clifford
system $\{P_0,\cdots,P_4\}$ on $\mathbb{R}^{16}$ with
$P_0\cdots P_4\neq\pm I_{16}$, it can be extend to a symmetric Clifford system
$\{P_0,\cdots,P_4,P_5\}$ on $\mathbb{R}^{16}$ indeed.
\end{rem}

\section{Isoparametric foliation and harmonic map}
This section will be concerned with harmonic maps and their energy-stability via isoparametric focal maps. We will prove Theorem \ref{eigenmap} and Proposition \ref{focal map}
on harmonic maps, and then investigate the stability of these harmonic maps.
For convenience, we begin with recalling the following basic fact.
\begin{lem}\cite{EL78}\label{tension field}
{\itshape For Riemannian manifolds $M$, $N$ and $P$, let $f$ be a smooth map from $M$ to $N$,
and $i$ an isometric immersion from $N$ into $P$. Define $F=i\circ f: M \rightarrow P$.
Then $f$ is harmonic if and only if the tension field of $F$ is
normal to $N$.}
\end{lem}
We are now ready to prove Theorem \ref{eigenmap}.

\noindent\textbf{Proof of Theorem \ref{eigenmap}:}
\begin{proof}
We only consider part (1) of this theorem. Let $\{P_0,\cdots,P_m\}$ be a symmetric Clifford system
on $\mathbb{R}^{2l}$. For $0\leq i \leq m-1$, the focal maps $\phi_{\pm \frac{\pi}{4}}: M_{i+1} \rightarrow
\mathcal{U}_{\pm 1}=SE_{\pm}(P_{i+1})$ given by
$$\phi_{\pm \frac{\pi}{4}}(x)=\frac{1}{\sqrt{2}}(x\pm P_{i+1}x),~x \in M_{i+1}$$
are smooth submersions, due to Proposition \ref{tangent map}. Since $M_{i+1}$ is
minimal in $S^{2l-1}(1)$, by using Takahashi Theorem, we get
$\triangle x=-(2l-i-3)x$. It follows that
$$\triangle \Phi_{\pm}=-(2l-i-3)\Phi_{\pm},$$
where $\Phi_{\pm}=i_{\pm}\circ \phi_{\pm\frac{\pi}{4}}$ and
$i_{\pm}: \mathcal{U}_{\pm 1}=SE_{\pm}(P_{i+1})\rightarrow E_{\pm}(P_{i+1})\cong\mathbb{R}^l$ are
inclusion maps. Now, the proof follows from Lemma \ref{tension field}.
\end{proof}

\begin{rem}\label{harmonic morphism}
By definition, a smooth map $f$ between Riemannian manifolds $M$ and $N$ is called a
\emph{harmonic morphism} if the pull back of any local harmonic function on $N$ by $f$ is also a local harmonic function
on $M$. It is well known that a smooth map is a harmonic morphism if and only if
it is  simultaneously harmonic and weakly horizontally conformal (see, for example, \cite{EL78}).
In particular, a submersive harmonic morphism should be a horizontally conformal map.
According to Proposition \ref{tangent map}, the eigenmaps determined by Theorem \ref{eigenmap} are
not harmonic morphisms.
\end{rem}
\begin{rem}
Recall that a harmonic map $f$ is called (energy) \emph{stable} if every second variation of the
energy functional at $f$ is nonnegative. Due to \cite{Le82} and
\cite{Pe84}, for $n\geq 3$, any stable harmonic map from any compact Riemannian manifold $M^m$ to $S^{n}(1)$
is constant. Thus, for $l\geq4$, the eigenmaps constructed in Theorem \ref{eigenmap}
are unstable. Moreover, for $l=3$, the eigenmaps we constructed are also unstable,
because that any stable harmonic map from any compact Riemannian manifold $M^m$ to $S^{2}(1)$ is a
harmonic morphism by one of the main results in \cite{Ch96}.
\end{rem}
Next, we deal with the case of isoparametric hypersurfaces in unit spheres.

\noindent\textbf{Proof of Proposition \ref{focal map}:}
\begin{proof}
Given an isoparametric hypersurface $M^n$ in $S^{n+1}(1)$ and $M_+$, $M_-$ the focal submanifolds with codimension $m_1+1$, $m_2+1$,
respectively. Choose a smooth field
$\xi$ of unit normals to $M$.  To prove this proposition, it is sufficient to consider one
focal map $\varphi:M\rightarrow M_+$. In this case, for each $x\in M\subset S^{n+1}(1)$, $\varphi(x)=\cos{\theta}x +\sin{\theta}\xi(x)$
for certain $\theta$. Due to Lemma \ref{tension field}, we need to compute the Laplacian of $F$ on $M$, where
$F=i\circ \varphi$ and $i:M_+\rightarrow \mathbb{R}^{n+2}$ is the inclusion map.
Using the fact that $M$ has constant mean curvature in $S^{n+1}(1)$ and the Codazzi equation, we compute directly and get
\begin{equation}\label{ab}
\left\{ \begin{array}{ll}
\triangle x=  -nx+H\xi,\nonumber\\
\triangle \xi=Hx-|B|^2\xi,\nonumber
\end{array}\right.
\end{equation}
where $B$, $H$ and $|B|^2$ are second fundamental form, mean curvature with respect to $\xi$ and squared norm of
the second fundamental form for the isoparametric hypersurface $M$ in $S^{n+1}(1)$, respectively.
Thus $\triangle F=(-\cos{\theta}n++\sin{\theta}H)x+(\cos{\theta}H-\sin{\theta}|B|^2)\xi$, which is normal to $M_+$.
It follows from Lemma \ref{tension field} that $\varphi$ is harmonic as desired.
\end{proof}
\begin{rem}
It is interesting that each harmonic map constructed in Proposition \ref{focal map} has constant energy density
everywhere, the proof of which depends on the characterization of tangent map of the focal map (c.f. pp.245 in \cite{CR85}).
\end{rem}

\begin{examp}
As we mentioned in the introduction, Cartan classified all isoparametric hypersurfaces in unit spheres with three
distinct principal curvatures (see, for example, pp.296-297 in \cite{CR85}). More precisely, such an isoparametric hypersurface must
be a tube of constant radius over a standard Veronese embedding of a projective plane $\mathbb{F}P^2$
into $S^{3m+1}(1)$, where $\mathbb{F}=\mathbb{R}$, $\mathbb{C}$, $\mathbb{H}$ (quaternions), $\mathbb{O}$ (Cayley numbers) for
$m = 1, 2, 4, 8$, respectively. Let $f: S^{3m+1}(1)\rightarrow \mathbb{R}$ be the restriction to $S^{3m+1}(1)$ of
the corresponding Cartan-M\"{u}nzner polynomial. Then $M^{3m}=f^{-1}(0)$ is the minimal isoparametric hypersurface with three distinct
constant principal curvatures $\cot\frac{\pi}{6}, \cot\frac{\pi}{2}, \cot\frac{5\pi}{6}$ of the same multiplicity $m$ with respect to
$\xi=\nabla f/|\nabla f|$, where $\nabla f$ is the gradient of $f$ on $S^{3m+1}(1)$,
and $M_{\pm}=f^{-1}(\pm 1)$ is isometric to $\mathbb{F}P^2$. Define a focal map
$\varphi_{\pi/2}:M\rightarrow M_{-}$ by
$$\varphi_{\pi/2}(x)=
(\cos{\pi/2})x +(\sin{\pi/2})\xi=\xi \quad for\quad x\in M\subset S^{3m+1}(1).$$
It follows from a direct calculation that the focal map
$\varphi_{\pi/2}:M\rightarrow M_{-}$  is a horizontally conformal submersion. In fact(c.f.\cite{GT13}),
$$|(\varphi_{\pi/2})_*(X)|=\sqrt{3}|X|, \forall X\in (\mathrm{Ker}(\varphi_{\pi/2})_*)^{\perp}\subset TM.$$
Moreover, the fibers of $\varphi_{\pi/2}$ are all totally geodesic in $M$. In one word, these facts show that
the focal map $\varphi_{\pi/2}:M\rightarrow M_{-}$ is a harmonic morphism by Proposition \ref{focal map} and
Remark \ref{harmonic morphism}.
\end{examp}
For isoparametric hypersurfaces in unit spheres of OT-FKM type, more harmonic maps are constructed as follows.
Let $\{P_0, \cdots, P_m\}$ be a symmetric Clifford system on $R^{2l}$ as before. Then it defines a family of isoparametric
hypersurfaces of OT-FKM type and two focal submanifolds $M_+$, $M_-$ in $S^{2l-1}$ of codimension
$m+1$, $l-m$, respectively. For any $P\in \Sigma(P_0,...,P_m)$,
set $\xi=Px$ for $x\in M_+$, which is a global unit normal vector field on $M_+$ in $S^{2l-1}(1)$. Associated with $\xi$, we define two maps
   $$ \phi: M_+ \to   M_-,~\psi: M_+ \to   M $$
by $\phi(x)=\frac{1}{\sqrt{2}}(x+\xi)$, and $\psi(x)=\cos{t}x+\sin{t}\xi$, where $M$ is any isoparametric hypersurface in the family and $t=\mathrm{dist}(M_+, M)$, the spherical
distance between $M_+$ and $M$.
Clearly the maps are well defined.
\begin{prop}\label{M_+}
Both of the maps $ \phi: M_+ \to   M_-$ and $\psi: M_+ \to   M $ are harmonic maps.
\end{prop}
\begin{proof}
The proof is similar to that of Theorem \ref{eigenmap}. By Takahashi Theorem and Lemma \ref{tension field}, it follows
that $\phi$ and $\psi$ are harmonic.
\end{proof}
\begin{rem}
It is not difficult to prove that $\phi$ and $\psi$ are not harmonic morphisms.
Moreover, $\psi$ is a section of the focal map (fibration) from $M$ to $M_+$
and the map $\psi: M_+ \to M $ is a smooth embedding.
\end{rem}
Whenever the Clifford system $P_0,\cdots,P_m$ can be extended to a Clifford system $P_0,\cdots, P_m$, $P_{m+1}$ on $\mathbb{R}^{2l}$,
one can define $\eta=P_{m+1}x$ for $x\in M_-$ in $S^{2l-1}(1)$. In fact, $\eta$ is a global unit normal vector field on $M_-$.
Associated with $\eta$, we can also define two maps
$$ \tilde{\phi}: M_- \to   M_+,~\tilde{\psi}: M_- \to   M $$
by $\tilde{\phi}(x)=\frac{1}{\sqrt{2}}(x+\eta)$, and $\tilde{\psi}(x)=\cos{s}x+\sin{s}\eta$, where $s=\mathrm{dist}(M_-, M)$, the spherical
distance between $M_-$ and $M$.
\begin{prop}\label{M_-}
Both of the maps $ \tilde{\phi}: M_-\to   M_+$ and $\tilde{\psi}: M_- \to   M $ are harmonic maps.
\end{prop}
\begin{proof}
The proof is analogous to that of Proposition 4.1 and is omitted.
\end{proof}
\begin{rem}
It is not difficult to prove that $ \tilde{\phi}$ and $\tilde{\psi}$ are not harmonic morphisms.
Moreover, the map $\tilde{\psi}: M_- \to   M $ is a smooth embedding. In fact, $\tilde{\psi}$ is a section
of the focal map (fibration) from $M$ to $M_-$.
\end{rem}

We now study the stability of harmonic maps. Recall that a compact Riemannian manifold $M$ is called harmonically unstable,
if there exists neither nonconstant
stable harmonic map from $M$ to any Riemannian manifold nor from any compact Riemannian manifold to $M$
(c.f. \cite{Oh86}). A significant result states that if $M$ is harmonically unstable then $\pi_1(M)=0$ and $\pi_{2}(M)=0$.
For minimal submanifolds in unit spheres, Ohnita obtained the following elegant result.
\begin{thm}\cite{Oh86}
{\itshape Let $M$ be an m-dimensional closed minimal submanifold in a unit sphere $S^{n}(1)$. If the Ricci
curvature $\rho$ of $M$ satisfies $\rho>m/2$, then $M$ is harmonically unstable.}
\end{thm}
Using this theorem, Ohnita also investigated the harmonically unstability of minimal isoparametric hypersurfaces in
unit spheres. As mentioned before, the focal submanifolds of isoparametric hypersurfaces are minimal in the unit sphere. Hence
we get the following proposition by applying Ohnita's theorem.
\begin{prop}\label{stable}
Given a symmetric Clifford system $\{P_0,...,P_m\}$ on $\mathbb{R}^{2l}$ and consider the corresponding
isoparametric hypersurface of OT-FKM type in $S^{2l-1}(1)$ with $g=4$ and multiplicities $(m_1, m_2)=(m, l-m-1)$. For the focal submanifold $M_+$
of codimension $m+1$ in $S^{2l-1}(1)$, if $(m_1, m_2)\neq (1, 1), (1, 2), (2, 1), (2, 3), (4, 3), (5, 2), (6, 1)$ and
$(9, 6)$, then $M_+$ is harmonically unstable.
\end{prop}
\begin{proof}
According to \cite{TY12}, for each point $x\in M_+$ and any unit tangent vector $X\in T_xM_+$, the
Ricci curvature of $M_+$ is given by
$$\rho(X, X)=2(l-m-2)+2\sum_{0\leq\alpha<\beta\leq m}\langle X, P_{\alpha}P_{\beta}x\rangle^2.$$
Then the proposition follows from the formula above and Ohnita's theorem.
\end{proof}
\begin{rem}\label{stable2}
For the exceptional cases in the proposition above, we have

1). $(m_1, m_2)=(1, 1)$:  Since $M_+^3$ is diffeomorphic to $\mathrm{SO}(3)$ (see pp. 301 of \cite{CR85}), $\pi_1(M_+^3)=\mathbb{Z}_2$. Hence $M_+^3$ is not harmonically unstable.

2). $(m_1, m_2)=(1, 2)$: Since $M_+^5$ is diffeomorphic to the unit tangent bundle of $S^3$ (see pp. 301 of \cite{CR85}), which is diffeomorphic to $S^2\times S^3$,
$\pi_2(M_+^5)=\pi_2(S^2\times S^3)=\mathbb{Z}$. Hence $M_+^5$ is not harmonically unstable.

3). $(m_1, m_2)=(2, 1)$: Since $M_+^4$ is diffeomorphic to $(S^1\times S^3)/\mathbb{Z}_2$ (see pp. 303 of \cite{CR85}), which is in turn diffeomorphic to
$S^1\times S^3$, $\pi_1(M_+^4)=\pi_1(S^1\times S^3)=\mathbb{Z}$.
Hence $M_+^4$ is not harmonically unstable.

4). $(m_1, m_2)=(5, 2)$: Since $\pi_1(M_+^{9})=0$, by Hurwitz isomorphism, $\pi_2(M_+^{9})=H_2(M_+, \mathbb{Z})=\mathbb{Z}$ (c.f. \cite{Mu80}).
Hence $M_+^9$ is not harmonically unstable.

5). $(m_1, m_2)=(6, 1)$: Similar to the case 3), $\pi_1(M_+^8)=\mathbb{Z}$.
Hence $M_+^8$ is not harmonically unstable.

6). $(m_1, m_2)=(4, 3)$ and the homogeneous case: For each point $x\in M_+^{10}$ and any unit tangent vector $X\in T_xM_+^{10}$,
the Ricci curvature
$$\rho(X, X)=4+
2\sum_{0\leq\alpha<\beta\leq 4}\langle X, P_{\alpha}P_{\beta}x\rangle^2=6>10/2,$$
since
$\{P_{\alpha}P_{\beta}x~|~0\leq\alpha<\beta\leq 4\}$ is an orthonormal basis of $T_xM_+^{10}$
(c.f. \cite{QTY13}). Hence $M_+^{10}$ is harmonically unstable by Ohnita's Theorem.

There are still three cases we have not determined, i.e., $(m_1, m_2)=(2, 3)$, $(4, 3)$ and the inhomogeneous case, or $(9 ,6)$.
\end{rem}

\begin{rem}
As a result of Proposition \ref{stable} and Remark \ref{stable2}, there exist unstable harmonic maps among the ones
we constructed in Proposition \ref{focal map}, \ref{M_+} and \ref{M_-}.
\end{rem}
\section{Counterexamples to Leung's conjectures}
This section will use the expansion formula of Cartan-M\"{u}nzner polynomial
and the isoparametric triple system to prove Theorem \ref{counter example}, providing infinitely
 many counterexamples to two conjectures of Leung \cite{Le91} on minimal submanifolds in unit spheres.

\noindent\textbf{Proof of Theorem \ref{counter example}:}
\begin{proof}
Let $M^n$ be an isoparametric hypersurface in $S^{n+1}(1)$ with $g=4$ and
multiplicities $(m_1, m_2)$,  and denote by $M_+$ and $M_-$ the focal
submanifolds of $M^n$ in $S^{n+1}$ with dimension $m_1+2m_2$ and $2m_1+m_2$
respectively. Note $n=2(m_1+m_2)$. Assume $F$
is the associated Cartan-M\"{u}nzner isoparametric
polynomial of degree four so that $M_+$ is defined by $F^{-1}(1)\cap S^{n+1}(1)$.
To complete the proof of this theorem, we only need to consider $M_+$, since if $F$ is changed to $-F$,
$M_+$ is changed to $M_-$.
Given $x\in M_+$, choose an orthonormal basis $\xi_\alpha$, $\alpha=0,1,...,m_1$ for
the normal space of $M_+$ in $S^{2l-1}(1)$ at the point $x$. Let $A_{\alpha}$, $\alpha=0,1,...,m_1$,
be the corresponding shape operators.
For any vector $X\in T_xM_+$, one has
$|B(X,X)|^2 = \sum_{\alpha=0}^{m_1}\langle A_{\alpha}X, X\rangle^2,$
where $B$ is the second fundamental form of $M_+$ in $S^{2l-1}(1)$.

For our purpose, we first recall a formulation of the Cartan-M{\"u}nzner polynomial $F$ in terms of the second fundamental forms of the focal submanifolds, developed by Ozeki and Takeuchi (see pp. 52 of \cite{CCJ07} and also \cite{OT75}). For $x\in M_+$, and an orthonormal basis $\{\xi_\alpha~|~\alpha=0,1,...,m_1\}$ of the normal
space of $M_+$ in $S^{n+1}(1)$ at $x$, one can introduce the quadratic homogeneous polynomials
$p_{\alpha}(y):=\langle A_{\alpha}y, y \rangle,$
for $0\leq \alpha \leq m_1$, where $y$ is tangent to $M_+$ at $x$. The Cartan-M{\"u}nzner polynomial $F$ is related to $p_{\alpha}$ as follows,
\begin{eqnarray}\label{expansion formula}
F(tx+y+w)&=&t^4+(2|y|^2-6|w|^2)t^2+8(\sum_{\alpha=0}^{m_1}p_{\alpha}(y)w_{\alpha})t \nonumber\\
      & & +|y|^4-2\sum_{\alpha=0}^{m_1}(p_{\alpha}(y))^2+8\sum_{\alpha=0}^{m_1}q_{\alpha}(y)w_{\alpha} \nonumber \\
      & & +2\sum_{\alpha,\beta=0}^{m_1}\langle\nabla p_{\alpha}, \nabla p_{\beta}\rangle w_{\alpha}w_{\beta}-6|y|^2|w|^2+|w|^4 ,\nonumber
\end{eqnarray}
where the homogeneous polynomial of degree three, $q_{\alpha}(y)$, are the components of the third fundamental form of $M_+$, and $w=\sum_{\alpha=0}^{m_1}w_{\alpha}\xi_{\alpha}$.

By the expansion formula above, we observe that for any $X\in T_xM_+$
$$F(X)=|X|^4-2\sum_{\alpha=0}^{m_1} (p_{\alpha}(X))^2 = |X|^4-2|B(X, X)|^2.$$
Hence, $|B(X, X)|^2=\frac{|X|^4-F(X)}{2}$.

Next, we will give an investigation into the possible value of $|B(X, X)|^2$, from which Theorem \ref{counter example} follows immediately.
To do it, we use the isoparametric triple system introduced by Dorfmeister and Neher, following the way in \cite{Im08}.
Let $x'$ be a unit vector normal to the tangent space $T_xM_+$ in $T_xS^{n+1}$.
Then the great circle $S$ through $x$ and $x'$ intersects the isoparametric hypersurface and two focal submanifolds orthogonally at each intersection point. The set $S\cap M_+$ consists of the four points $ \pm x$ and $\pm x'$, and the set $S\cap M_-$ consists of the four points $ \pm y$ and $\pm y'$, where $\sqrt{2}x=y-y'$ and $\sqrt{2}x'=y+y'$. There are orthogonal Peirce decompositions
$$\mathbb{ R}^{2m_1+2m_2+2}=\mathrm{Span}\{x\} \oplus V_{-3}(x) \oplus V_1(x)=\mathrm{Span}\{y\} \oplus V_{3}(y) \oplus V_{-1}(y),$$
where $V_{-3}(x)=T_x^\perp M_+$, the normal space of $M_+$ in $S^{n+1}(1)$ at $x$, $V_1(x)=T_xM_+$, $V_{3}(y)=T_y^\perp M_-$, the normal space of $M_-$ in $S^{n+1}(1)$ at $y$, and $V_{-1}(y)=T_yM_-$, the so-called\emph{ Peirce spaces}. Furthermore, as one of the main results in \cite{Im08}, Immervoll gave a more subtle orthogonal decomposition as
$$ \mathbb{R}^{2m_1+2m_2+2}=\mathrm{Span}(S) \oplus V'_{-3}(x) \oplus V'_{-3}(x') \oplus V'_{3}(y) \oplus V'_3(y')$$
where the subspaces $V'_{-3}(x)$, $V'_{-3}(x')$, $V'_{3}(y)$ and $V'_{3}(y')$ are defined by $V_{-3}(x)=\mathrm{Span}\{x'\} \oplus V'_{-3}(x)$, $V_{-3}(x')=\mathrm{Span}\{x\} \oplus V'_{-3}(x')$,
$V_{3}(y)=\mathrm{Span}\{y'\} \oplus V'_{3}(y)$ and $V_{3}(y')=\mathrm{Span}\{y\} \oplus V'_{3}(y')$.
It follows from the two decompositions above that
$$ T_xM_+ = V'_{-3}(x') \oplus V'_{3}(y) \oplus V'_3(y').$$
Now taking a unit vector $X_1$ in $V'_3(y)$, we see that $X_1\in M_-$. Actually,
$S^{2l-1}(1)\cap V_3(y)\subset M_-$. Similarly, any unit vector $X_1$
in $V'_3(y')$ also belongs to $M_-$. Hence, for any unit vectors $X_1\in V'_3(y)\oplus V'_3(y')$, we have $F(X_1)=-1$, and thus
$|B(X_1,X_1)|^2=\frac{1-F(X_1)}{2}=1.$
On the other hand, for any unit vector $X_0\in V'_{-3}(x')$, we can see that $X_0\in M_+$ and $F(X_0)=1$. Therefore,  $|B(X_0,X_0)|^2=\frac{1-F(X_0)}{2}=0$.
We have proved that $M_+^{m_1+2m_2}$ is a minimal submanifold in $S^{n+1}(1)$ with $\sigma(M_+)=1$.

Lastly, according to \cite{Mu80}, the cohomology ring of $M_+$ is different from that of $S^{m_1+2m_2}$,
and thus $M_+$ is not homeomorphic to $S^{m_1+2m_2}$.

Now, the proof is complete.
\end{proof}
\begin{ack}
The authors would like to thank Professors Y. Ohnita and E. Loubeau for useful comments on the stability of harmonic maps, and
thank Prof. Jiagui Peng for helpful advices during the preparation of the paper. Thanks are also due to Prof. Weiping Zhang for presenting the second author with the book \cite{ER93} from Paris. Finally, the authors are very grateful to the referees for useful comments and suggestions.
\end{ack}

\end{document}